\documentclass[reqno,a4size]{amsart}
\usepackage{graphicx, enumerate,amssymb,amscd}
\newtheorem{thm}{Theorem}[section]
\newtheorem{cor}[thm]{Corollary}
\newtheorem{lem}[thm]{Lemma}
\newtheorem{prop}[thm]{Proposition}
\newtheorem{result}[thm]{Result}
\theoremstyle{definition}
\numberwithin{equation}{section}
\newcommand{\norm}[1]{\left\Vert#1\right\Vert}
\newcommand{\abs}[1]{\left\vert#1\right\vert}
\newcommand{\rl}{{\mathbb{R}}}
\newcommand{\cx}{{\mathbb{C}}}
\newcommand{\dist}{{\mathrm { dist}}}
\newcommand{\lo}{{\rm Lip}_{\mathcal O}}
\DeclareMathOperator*{\res}{res}

\newcommand{\psh}{plurisubharmonic }
\newcommand{\dbar}{\overline{\partial}}
\newcommand{\bd}{\mathsf{b}}
\newcommand{\tmop}[1]{\ensuremath{\operatorname{#1}}}
\renewcommand{\Re}{\tmop{Re}}
\renewcommand{\Im}{\tmop{Im}}

\title{Function theory and  holomorphic maps on symmetric products of planar domains}
\author{Debraj Chakrabarti}
\address{TIFR Centre for Applicable Mathematics, Sharada Nagar, Chikkabommasandra, Bengaluru-560065, India}
\email{debraj@math.tifrbng.res.in}
\author{Sushil Gorai}
\thanks{Sushil Gorai was supported in part by an INSPIRE Faculty Fellowship awarded by DST, Government of India.}
\address{Stat-Math Unit, Indian Statistical Institute, 8th Mile, Mysore Road, Bengaluru-560059, India}
\email{sushil@isibang.ac.in}
\begin{document}
\begin{abstract} We show that the $\dbar$-problem
is globally regular on a domain in $\cx^n$, which is the $n$-fold symmetric product  of a smoothly bounded planar domain. 
Remmert-Stein type theorems are proved for proper holomorphic maps between equidimensional symmetric products and proper
holomorphic maps from cartesian products to symmetric products.
It is shown  that  proper holomorphic maps between equidimensional symmetric products of smooth planar domains are smooth up to the boundary.\end{abstract}

\maketitle
\section{Introduction}
\subsection{Symmetric Products}
Let $n\geq 2$ be an integer. Given an object $X$, a basic construction in mathematics is to form its $n$-fold cartesian product $X^n$ with 
itself, the set of ordered $n$-tuples $(x_1,x_2,\dots, x_n)$ of elements  $x_i\in X$.
One can equally consider the space $X^{(n)}$ of {\em unordered} $n$-tuples of elements of $X$, called the $n$-fold {\em symmetric product} of $X$.
If $\sigma$ is a bijection of the set
 $I_n=\{1,2,\dots, n\}$ with itself, the tuples $(x_1,x_2,\dots, x_n)$ and $(x_{\sigma(1)}, x_{\sigma(2)},\dots, x_{\sigma(n)})$ are to be considered identical  in 
$X^{(n)}$.    Unlike in the case of the cartesian product, the smoothness of $X$ 
is usually not inherited by $X^{(n)}$.   For example, if $X$ is a complex manifold, the symmetric product $X^{(n)}$ is in general only a { complex space},
which locally looks like a complex analytic set, possibly with singularities. 

If  $U$ is a Riemann surface, then the symmetric product $U^{(n)}$ is a complex manifold of dimension $n$ in a natural way.
In particular, if  $U$ is a domain in the complex plane, the symmetric product $U^{(n)}$  
is in fact  biholomorphic to a domain $\Sigma^n U$ in $\cx^n$, which may be 
 constructed in the following way.
For $k\in I_n=\{1,2,\dots, n\}$, denote by $\pi_k$ the $k$-th
{ elementary symmetric polynomial} in $n$ variables.  
Let $\pi:\cx^n\to \cx^n$ be the polynomial map given by $ \pi=(\pi_1,\dots,\pi_n)$,
referred to as the {\em symmetrization} map on $\cx^n$.  We then define $ \Sigma^nU = \pi(U^n)$,
where $U^n\subset\cx^n$ is the $n$-fold cartesian product of the domain $U$.  Then $\Sigma^n U$ is biholomorphic to the
$n$-fold symmetric product $U^{(n)}$ (see Section~\ref{sec-symmetrization} below.)

In this paper, we study the domains $\Sigma^n U$ from the point of view of the classical theory of functions and mappings of several complex variables.
For any domain $U\subset \cx$,  $\Sigma^n U$ is a pseudoconvex domain in $\cx^n$. The main interest is therefore in the study of boundary properties
of functions and mappings. The boundary $\bd \Sigma^n U$ of $\Sigma^n U$ is a
singular  hypersurface in $\cx^n$, with a Levi-flat smooth part. Even when $U$ has smooth boundary,
the boundary $\bd \Sigma^n U$ is not Lipschitz, 
a fact which was shown to us recently by {\L}ukasz Kosi\'{n}ski in the case $n=2$, and from which the general case follows.
In spite of having such non-smooth boundary, it turns out that the complex-analytic properties of $\Sigma^n U$ are in many ways
similar to  those of the cartesian product $U^n$.

\subsection{Main Results} Recall that 
the $\dbar$-problem  on a domain consists of solving the equation $\dbar u=g$, where $g$ is a given $(p,q)$-form with $q>0$ such 
that $\dbar g=0$, and $u$ is an unknown $(p,q-1)$-form. Since  $\Sigma^n U$ is pseudoconvex, 
the $\dbar$-problem always has a solution on it. We  also have interior regularity: if the datum $g$ has coefficients which are 
$\mathcal{C}^\infty$-smooth, there is a solution $u$ which also has smooth coefficients (see e.g. \cite{hor}.)
The interesting question  regarding $\Sigma^n U$ is therefore that of boundary regularity of the $\dbar$-problem. 
Recall that on a domain $\Omega\Subset\cx^n$, the $\dbar$-problem is said to be {\em globally regular} if for any $\dbar$-closed form
$g\in \mathcal{C}^\infty_{p,q}(\overline{\Omega})$   with $q>0$, there exists a form 
$u\in \mathcal{C}^\infty_{p,q-1}(\overline{\Omega})$ such that $\dbar u= g$,
where  $\mathcal{C}^\infty_{p,q}(\overline{\Omega})$ denotes the space of forms of 
degree $(p,q)$ with coefficients which are smooth up to the boundary. Our first result is the following:
\begin{thm}\label{thm-dbar}  
If $U\Subset \cx$ is a domain with $\mathcal{C}^\infty$-smooth boundary in the plane,  the $\dbar$-problem is globally regular on $\Sigma^nU$.
\end{thm}

For smoothly bounded domains, a classical approach to the study of the regularity of the $\dbar$-problem  is to reduce it to the 
study of the $\dbar$-Neumann problem for the complex Laplace operator (see \cite{straube, chen-shaw, fk} for details.) Indeed, from Kohn\rq{}s 
weighted theory of the $\dbar$-Neumann operator (see \cite{kohn}) it follows that on a smoothly bounded pseudoconvex domain in $\cx^n$, the $\dbar$-problem
is globally regular.  However, as we show in 
Proposition~\ref{prop-unonlip} below, the domain $\Sigma^n U$ is not Lipschitz, even though $U$ may have $\mathcal{C}^\infty$ boundary.
Consequently many of the basic tools of the theory of the $\dbar$-Neumann problem, including the crucial method of 
\lq\lq{}elliptic regularization\rq\rq{}, do not apply to  $\Sigma^n U$. Furthermore, the presence of complex hypersurfaces in the smooth 
part of $\bd\Sigma^n U$ means that the $\dbar$-Neumann operator on $\Sigma^n U$ is non-compact, and one would even expect it  to be  highly non-regular
(cf. \cite{ehsani}.)

For non-smooth pseudoconvex domains, the global regularity of the $\dbar$-problem is a  subtle matter.
It is well-known that on the Hartogs Triangle, the non-Lipschitz domain in $\cx^2$  given by  $\{\abs{z_1}<\abs{z_2}<1\}$, the $\dbar$-problem is not 
globally regular (see \cite{chaumat-chollet,chak-shaw}.) However the geometry of the Hartogs Triangle is very different from that of symmetric products
 in that the Hartogs Triangle does not
admit a basis of Stein neighborhoods.  The presence of a Stein neighborhood basis with appropriate geometric properties 
(``$s-H$"-convexity in the terminology of \cite{cc}) is crucial in the proof of 
Theorem~\ref{thm-dbar}. Note that though  the $\dbar$-Neumann operator on 
$\Sigma^n U$ is non-compact,  the $\dbar$-problem is globally regular.
The situation on $\Sigma^nU$ is therefore analogous to that on cartesian products (see \cite{prod, michelshaw}.)

Next, we  study  proper holomorphic mappings of symmetric products.  It is natural to study mappings between symmetric 
and cartesian products.  When $U=V=\mathbb{D}$,
the unit disc in the plane, this was done by Edigarian and Zwonek in \cite{edi1,ediz1}, using the classical method of Remmert and Stein 
(see \cite{remmertstein,nar}.) Generalizing this, we have the following:
\begin{thm}\label{thm-proper}Let $U_1,\dots, U_n$ be bounded domains in the  complex plane and let $G=U_1\times\dots\times U_n$ be their cartesian product.  
Let $f:G\to\Sigma^n V$ be a proper holomorphic map. Then there are proper holomorphic maps $g_j:U_j\to V$  such that 
\[ f = \pi\circ (g_1\times \dots\times g_n),\]
where by definition, for $z\in G$, we have  $(g_1\times \dots\times g_n)(z)=(g_1(z_1),\dots, g_n(z_n))$.
\end{thm}
This can compared with the fact that  a  proper 
holomorphic mapping between equidimensional cartesian product domains splits as product of mappings in the factors (see \cite[page 77]{nar}.)

Theorem~\ref{thm-proper} leads to the classification of proper holomorphic  maps between equidimensional symmetric products: for any mapping 
$\psi:U\to V$ of planar domains,  let $\Sigma^n \psi:\Sigma^n U\to \Sigma^n V$ be the unique map such that
\begin{equation}\label{eq-sigmanpsi} (\Sigma^n\psi)(\pi(z_1,\dots, z_n))= \pi\left(\psi(z_1),\dots,\psi(z_n)\right).\end{equation}
(see  Section~\ref{sec-induced} below.) We have the following:
\begin{cor}\label{cor-proper}
$U,V$ be bounded planar domains, and let $\Phi:\Sigma^n U\to \Sigma^n V$ be a proper holomorphic map. Then there is a proper holomorphic $\varphi:U\to V$ 
such that $\Phi=\Sigma^n \varphi$.
\end{cor}
Therefore, the study of boundary behavior of proper maps between equidimensional symmetric products is reduced to the study of boundary behavior 
of maps of the form $\Sigma^n\psi$ as defined in \eqref{eq-sigmanpsi}.  Study of regularity of $\Sigma^n \psi$ is a major thrust of this paper.
For $k$ a non-negative integer or $\infty$, and a domain $\Omega\subset \cx^m$,  denote by 
$\mathcal{A}^k(\Omega)$ 
the space $\mathcal{O}(\Omega)\cap \mathcal{C}^k(\overline{\Omega})$ of holomorphic functions  on $\Omega$ which are $\mathcal{C}^k$-smooth up to the boundary. A map is of class
$\mathcal{A}^k(\Omega)$ if each component is. We prove the following result:
\begin{thm}\label{thm-akn}
 Let $k$ be a non-negative integer or $\infty$,  let $U, V$ be bounded domains in $\cx$ with rectifiable boundaries, and let $\varphi: U\to V$ be a holomorphic map such that 
$\varphi\in\mathcal{A}^{kn}(U)$. Then the induced map $\Sigma^n\varphi: \Sigma^n U\to \Sigma^n V$ is in $\mathcal{A}^k(\Sigma^n U)$.
Further, if $\varphi:U\to V$ is a holomorphic map which is not of class $\mathcal{A}^{kn}(U)$, then the induced map is not of class $\mathcal{A}^k(\Sigma^n U)$.
\end{thm}
Note the  loss of smoothness from $kn$ times differentiability at the boundary
 to $k$ times differentiability at the boundary in going from $\varphi$ to $\Sigma^n\varphi$. This is not 
unexpected given the serious branching behavior of the symmetrization map $\pi:\cx^n\to \cx^n$.
It would be interesting to obtain precise information on the smoothness of the map $\Sigma^n\varphi$ when $\varphi$ is of class $\mathcal{A}^r(U)$ where
$r$ is not an integer  multiple of $n$, or for $\varphi$ in H\"{o}lder classes. Combining Theorem~\ref{thm-akn}  with classical results on boundary regularity of 
proper holomorphic maps in one variable, we obtain information on the boundary regularity of  proper maps between symmetric products, e.g.,
\begin{cor}\label{cor-bdryreg}
Let $U,V\Subset\cx$ be domains with $\mathcal{C}^\infty$-smooth boundaries, and let $\Phi:\Sigma^nU\to \Sigma^nV$ be a proper holomorphic map. Then
$\Phi$ extends as a $\mathcal{C}^\infty$-smooth map from $\Sigma^n\overline{U}$ to $\Sigma^n\overline{V}$.
\end{cor}
\subsection{General remarks}Topological symmetric products, under the name {\em unordered configuration spaces} have been studied extensively. 
One usually restricts attention to the \lq\lq{}unbranched part\rq\rq{} of such a space, i.e., one considers the quotient $(X^n\setminus \Delta)/S_n$,
where $\Delta$ consists of those points $(x_1,\dots, x_n)$ of $X^n$, where at least two of the coordinates $x_i$ and $x_j$ are the same. The fundamental
group of such spaces are an important subject of study under the name \lq\lq{}Braid Groups\rq\rq{} (see \cite{magnus}.)

Complex symmetric products have many applications in the theory of complex spaces,
their holomorphic mappings and correspondences (see, e.g., \cite{whitney,bedfordbell}.)
For a compact Riemann surface $C$, the points of $\Sigma^n C$ are in bijection with the  effective divisors on $C$ of degree $n$, and
the $n$-th Abel-Jacobi map maps $\Sigma^n C$ into the Jacobian variety of $C$ (birationally if $n$ is equal to the genus of $C$, see \cite[p. 18 ff.]{ac}.)

Denoting by  $\mathbb{D}$  the unit disc in the plane, the domain 
$\Sigma^n \mathbb{D}$ (under the name {\em symmetrized polydisc}, or for $n=2$, {\em symmetrized bidisc})
has been studied extensively.  It arises naturally in certain problems in control theory and the related spectral Nevanlinna-Pick interpolation
problem (see \cite{ay1}, and the subsequent work deriving from it.)  Subsequently, the domains
 $\Sigma^n \mathbb{D}$  have been studied from the point of 
complex geometry (see e.g. \cite{ay2,costara,costara2,misra, edi1,ediz1,ay3, gamma} etc.)  In view of Lempert's theorem on the identity of the
Carath\'{e}odory and Kobayashi metrics on convex domains,  the most striking feature is the identity of these 
metrics on  $\Sigma^2 \mathbb{D}$, though $\Sigma^2\mathbb{D}$ cannot be exhausted by an increasing sequence of
 domains biholomorphic to convex domains. This however seems to be  a special property of the symmetrized bidisc.
 
 Our work here  considers symmetric products of general planar domains, with particular reference to boundary 
 regularity, and generalizes  properties of symmetrized polydiscs.

The paper is organized as follows. In Section~\ref{sec-elem} below, we consider some general properties of symmetric products as a preliminary 
to the proof of the results stated above. We prove a formula (see \eqref{eq-ppsi}) for the induced map $\Sigma^n\varphi$ which plays a crucial role 
in the sequel, as well as study the branching behavior of the  symmetrization mapping  $\pi:\cx^n\to \cx^n$. Section~\ref{sec-cauchy} is devoted to 
obtaining some estimates in spaces of holomorphic functions smooth up to the boundary on an integral operator $\mathsf{J}$ (see \eqref{eq-amdef}.) This 
operator arises in the proof of Theorem~\ref{thm-akn}, and through Theorem~\ref{thm-akn} in the proof of other results of this paper.  The estimates are then used  in the following section to prove 
Theorem~\ref{thm-akn}. In Section~\ref{sec-thm-dbar},  Theorem~\ref{thm-dbar} is proved first for a model domain using a construction of a Stein 
neighborhood basis. Conformal mapping and a pullback argument gives the general result. Similarly, a pullback method is used to show that the boundary of
a symmetric product is not Lipschitz. The last section is devoted to the proof of Theorem~\ref{thm-proper} and deducing Corollaries~\ref{cor-proper} and \ref{cor-bdryreg}.

Starting with Section~\ref{sec-cauchy}, we use the \lq\lq{}Variable Constant\rq\rq{} convention, so that $C$ or $C_j$ may denote different constants at different appearances.
The convention is standard with inequalities and estimates, but here we use it in the context of equalities as well, for constants which depend only on the dimension and domain.
\subsection{Acknowledgements} We would like to thank our colleagues at Bangalore for many interesting discussions. 
We would also like to thank W{\l}odzimierz Zwonek and  Armen Edigarian for helpful information and {\L}ukasz Kosi\'{n}ski
for telling us about his result on the boundary geometry  of the bidisc  (Lemma~\ref{lem-kos} below.)
Debraj Chakrabarti would like to thank  Diganta Borah for his invitation to visit IISER Pune, and colleagues there for useful comments and  discussions on 
the subject of this paper. Sushil Gorai would like to thank Stat-Math Unit, Indian Statistical Institute, Bangalore Centre for providing a congenial 
atmosphere for research.

We would also like to express our special gratitude to the Dean of  TIFR-CAM, Prof. Mythily Ramaswamy for facilitating our collaboration 
by providing  chauffeured cars to travel between our  institutes. Without her support this paper would never have been written. 

\section{Elementary properties of symmetric products}\label{sec-elem}

\subsection{The symmetrization map}\label{sec-symmetrization} Let $S_n$ denote the symmetric group of bijections of $I_n=\{1,\dots,n\}$. Given any set $X$
we can define an action of $S_n$ on the $n$-fold cartesian product $X^n$ by setting 
\begin{equation}\label{eq-sigmax}\sigma(x_1,\dots, x_n)= (x_{\sigma(1)},\dots, x_{\sigma(n)}).\end{equation} Then the symmetric 
product $X^{(n)}$ may be identified with the quotient $X^n/S_n$ consisting of orbits of this action. It follows that in a category where quotients by finite groups exist, 
we can define symmetric products. It is known that in the category of {\em complex spaces} (ringed spaces locally modelled on complex analytic sets,
see \cite{guro,whitney}) such quotients by finite groups do indeed exist (see \cite{cartan1,cartan2}.)
Therefore the symmetric product of a domain in 
$\cx^k$ is in general a complex space with singularities.  For a profound study of symmetric products in the complex setting 
see \cite{whitney}, especially Appendix V.

If $U$ is a domain in $\cx^k$ clearly the symmetric product $U^{(n)}$ sits inside the symmetric product $(\cx^k)^{(n)}$ as an open set. 
Therefore, to construct $U^{(n)}$, it suffices to construct 
the space $(\cx^k)^{(n)}$ as a complex space. One way to do this  ( \cite{cartan1, cartan2} or \cite[Chaper 5, section 5]{whitney})
is to consider $(\cx^k)^{(n)}$ as an affine algebraic variety and imbed it in $\cx^N$ (for large $N$). Indeed, the coordinate ring of $(\cx^k)^{(n)}$ is nothing but the ring of those
polynomials in $n$ variables on $\cx^k$ which are left invariant by the action of $S_n$ described in the previous section.  When $k=1$,  this ring is 
the ring of symmetric polynomials in $n$ variables. It is well known that this ring is freely generated by the elementary symmetric polynomials.
We recall that the elementary symmetric polynomials  $\{\pi_k\}_{k=1}^n$ can be characterized by the fact that
\begin{equation}\label{eq-viete}(t-z_1)(t-z_2)\dots(t-z_n)= t^n-\pi_1(z)t^{n-1}+\dots+ (-1)^{n}\pi_n(z),\end{equation}
and  are given explicitly by
\[ \pi_k(z_1,\dots,z_n)= \sum_{\substack{J\subset I_n\\\abs{J}=k}}\left(\prod_{j\in J} z_j\right)\\
= \sum_{1\leq j_1 \leq \dots \leq j_k \leq n} z_{j_1}\dots z_{j_k}.
\]
In particular, the symmetrization map $\pi=(\pi_1,\dots, \pi_n):\cx^n\to\cx^n$, gives a concrete realization of the quotient map 
$\cx^n\to \cx^{(n)}=\cx^n/S_n$.  We note a number of elementary properties of the map $\pi$:
\begin{enumerate}[(a)]
\item For $s\in \cx^n$, a point $z=(z_1,\dots,z_n)\in\cx^n$ is in $\pi^{-1}(s)$, if and only if
$z_1,\dots, z_n$ are roots of the polynomial $q(s;t)$, where
\begin{equation}\label{eq-q} q(s;t)=t^n-s_1t^{n-1}+\dots +(-1)^ns_n.\end{equation}
This is an immediate consequence of \eqref{eq-viete}.

\item From the above, it follows that  the map $\pi$ is surjective, thanks to the fundamental theorem of algebra. 

\item We can define the biholomorphic map identifying $\cx^n$ with $\cx^{(n)}$ in the following way. For an $s\in \cx^n$, let $(z_1,\dots, z_n)$ be the roots of 
the polynomial $q$ of equation \eqref{eq-q}. Then the tuple $(z_1,\dots, z_n)$ has no natural order, and hence may be considered to be a point in the
symmetric product $\cx^{(n)}$. This defines a map $\psi:\cx^n\to \cx^{(n)}$, $s\mapsto (z_1,\dots,z_n)$. This is easily seen to be a set-theoretic bijection, 
and using the complex structure on $\cx^{(n)}$ one can show that this identification is biholomorphic
(see \cite{whitney} for details.)  Similarly, for any subset $U$ of $\cx^n$, the set $\Sigma^n U$ can be  identified with $U^{(n)}$.

\item The map $\pi$ is open, i.e., if $\omega\subset\cx^n$ is open, then so is $\pi(\omega)$. This is a consequence of the classical 
fact that roots of a  monic polynomial  with complex coefficients depend continuously on the coefficients (cf. \cite[ Appendix V (Sec. 4)]{whitney})

\item By the previous part, if $U\subset \cx$ is a domain in $\cx$, the set $\Sigma^n U=\pi(U^n)$ is a domain in $\cx^n$. Furthermore, $\Sigma^n U$ is 
pseudoconvex: if  $\psi$ is a subharmonic exhaustion function of the domain  $U$ (which always exists), the function $\Phi$ given on $\Sigma^n U$ by
$\Phi(s) = \sum\limits_{ j=1}^n \psi(z_j),$
where $(z_1,\dots, z_n)$ are the roots of $q(s;t)=0$ is well-defined and a \psh exhaustion of $\Sigma^n U$.

\item The map $\pi$ is proper, i.e., the inverse image of a compact set is compact under $\pi$. Indeed, if $\abs{s}<r$, and $\abs{t}> \max(\sqrt{n}r,1)$,
we have
\[ \abs{q(s;t)}> \abs{t}^n - \sqrt{n} \abs{s}\abs{t}^{n-1}>0.\]
Consequently, the inverse image $\pi^{-1}(B(0,r))$ is contained in the ball $B(0,\max(\sqrt{n}r,1))$, and therefore $\pi$ is proper.
\end{enumerate}

We note here that if $k\geq 2, n\geq 2$, the space $(\cx^k)^{(n)}$ is {\em not} a smooth complex manifold. For example, the complex space
$(\cx^2)^{(2)}$ may be biholomorphcally identified with the singular complex hypersurface in $\cx^5$ with equation
\[ 4\left( z_3z_5 -z_4^2\right) = z_2^2 z_3 - 2 z_1 z_2 z_4 + z_1^2 z_5.\]
For details see \cite[pp. 78-79, especially equation 4.18]{olver}.

\subsection{The canonical stratification}
\label{sec-strata}

By a {\em partition} of a set $I$, we mean a collection of subsets of $I$  which are pairwise disjoint and whose union is $I$. The {\em length} of a partition
is simply the number of subsets of $I$ in the partition. For any positive integer $m$, let $I_m$ denote the $m$-element set $\{1,\dots, m\}$. 
We denote by $\mathcal{P}^n_k$ the set of partitions of $I_n$ of length $k$.

For each partition $P=\{P_1,\dots,P_k\}\in\mathcal{P}^n_k$ of $I_n$  of length $k$, we define 
\begin{equation}\label{eq-sp}
 S(P)=\left\{z\in \cx^n\left\vert \begin{split}\text{ there are {\em distinct} points}\; w_1,\dots,w_k\in\cx\;\; \text{such that  }\quad \\ \text{  for each $\ell\in I_k$, and  for each $i\in P_\ell$, we have }\;\; z_i=w_\ell\end{split}\right.\right\}.
\end{equation}

If we also let
\[ \overline{S}(P)=\left\{z\in \cx^n\left\vert \begin{split}\text{ there are points}\; w_1,\dots,w_k\in\cx\;\; \text{such that   for }\;\;\\ \text{  each $\ell\in I_k$, and  for each $i\in P_\ell$, we have }\;\; z_i=w_\ell\end{split}\right.\right\},\]
(so that $w_1,\dots, w_k$ are no longer assumed distinct)  then clearly $\overline{S}(P)$ is a linear subspace of $\cx^n$ of dimension $k$ with 
linear coordinates $(w_1,\dots,w_k)$. Then $S(P)$ is a dense open subset of this linear space, and the complement $\overline{S}(P)\setminus S(P)$ 
is an analytic set defined by  $\prod\limits_{j<k} (w_j-w_k)=0$. Consequently $S(P)$ is a connected complex manifold of dimension $k$ (indeed a domain of dimension $k$.)  We also define 
 \[
 {\mathsf V}_k=\bigcup_{P\in \mathcal{P}_k^n}S(P),
\]
where clearly the union is disjoint.
Then ${\mathsf V}_k$ is precisely the set of points $z\in \cx^n$, whose coordinates $(z_1,\dots,z_n)$ take $k$ distinct numbers as values.
Note also that each $S(P)$ is open in ${\mathsf V}_k$, and in fact is  a connected component of ${\mathsf V}_k$. 
The following fact will be used later
\begin{lem} \label{lem-localbiholo}For each $k\in I_n$, the restriction $\pi|_{{\mathsf V}_k}$ is a local biholomorphism.
\end{lem}
\begin{proof} Let $z^*\in {\mathsf V}_k$, and set $s^*=\pi(z^*)$. We need to find a neighborhood $D$ of $z^*$ in ${\mathsf V}_k$ such that $\pi|_D$ 
is a biholomorphism onto $\pi(D)$. To define $D$, first note that there is a partition $P\in \mathcal{P}^n_k$ such that  $z^*$  is in $S(P)$. Let $w_1,\dots,w_k$
be the natural coordinates on $S(P)$ and suppose that the natural coordinates of $z^*$ are $(w_1^*,\dots, w_k^*)$.  Since the points $w_1^*,\dots, w_k^*$ are
all distinct, we can find open discs $D_1,\dots, D_k$ in $\cx$, centered at $w_1^*,\dots,w_k^*$, such that their closures are pairwise disjoint.  We define the open subset
$D$ of $S(P)$ by setting
\[ D= \{z\in S(P)\colon \text{ for } j\in I_k, \text{ we have  } w_j\in D_j\}.\]
Then $D$ is an open neighborhood of the point $z^*$ in $S(P)$.  Let $G\subset\cx^n$ be given by
\[ G=\left\{s\in \cx^n\colon \text{ no root of $q(s;t)$ lies in $ \displaystyle{\bigcup\limits_{j=1}^k \bd D_j}$}\right\},\]
where $q$ is as in \eqref{eq-q}. Since the roots of a monic polynomial depend continuously on the coefficients, it follows that  the set $G$ is open. 
For  the partition $P=\{P_1,\dots, P_k\}$, and $j\in I_n$, there is  an $\mu(j)$ such that $j\in P_{\mu(j)}$.  For $s\in G$, define  for each $j\in I_n$,
\[ \gamma_j(s) = \frac{1}{2\pi i \abs{P_{\mu(j)}}} \int_{\bd D_{\mu(j)}} t \cdot\frac{q'(s;t)}{q(s;t)}dt,\]
where $q$ is as in \eqref{eq-q} and  $q'(s;t)$ is the partial derivative 
\begin{equation}\label{eq-qprime} \frac{\partial q}{\partial t}(s;t)=nt^{n-1}- (n-1)s_1 t^{n-2}+ \dots + (-1)^{n-1}s_{n-1}.\end{equation}
By differentiating with respect to $(\overline{s}_1,\dots,\overline{s}_n)$ under the integral sign, and noting that the integral on the right hand side is well defined if $s\in G$, we see that the tuple $\gamma=(\gamma_1,\dots,\gamma_n)$ defines a holomorphic map from $G$ to $\cx^n$.

Note that $\pi(D)\subset G$. For $s\in \pi(D)$, let $z\in D\subset S(P)$ be such that $\pi(z)=s$. For a given $j\in I_n$,
the integrand $ t \cdot\dfrac{q'(s;t)}{q(s;t)}$ has a pole at $t=z_j$ of order $\abs{P_{\mu(j)}}$ in the disc $D_{\mu(j)}$.  Applying the 
residue formula we see that $\gamma_j(s)=z_j$, i.e., $\gamma:\pi(D)\to D$ is the inverse of the map $\pi:D\to\pi(D)$. The result is proved.
\end{proof}
\begin{cor}\label{cor-v1}The mapping $\pi|_{{\mathsf V}_1}$ from ${\mathsf V}_1$ onto $\pi({\mathsf V}_1)$ is a biholomorphism.
\end{cor}
\begin{proof}Thanks to Lemma~\ref{lem-localbiholo}, we only need to show that the mapping $\pi$ is injective on ${\mathsf V}_1$. The set ${\mathsf V}_1$ consists of points 
such that all the coordinates are equal. Denoting the natural coordinate on ${\mathsf V}_1$ as $w$, we see that 
the image of the point $(w,\dots,w)\in {\mathsf V}_1$ is $\left(nw, \dfrac{n(n-1)}{2}w^2,\dots, w^n\right)$, which immediately shows that $\pi$ is
injective on ${\mathsf V}_1$.
\end{proof}
\subsection{The induced map $\Sigma^n \varphi$}\label{sec-induced}
 Let $\varphi:X\to Y$ be a mapping of sets. Then we can naturally define a map $\varphi^n:X^n\to Y^n$ by setting
\[ \varphi^n(x_1,\dots, x_n) = (\varphi(x_1),\dots,\varphi(x_n)).\]
The map $\varphi^n$ is equivariant with respect to the action \eqref{eq-sigmax} of $S_n$, i.e. for each $\sigma\in S_n$,  thought of as acting on $X^n$ or $Y^n$,
we have $\sigma\circ \varphi^n=\varphi^n\circ \sigma$. Then clearly,  $\varphi^n$ maps the $S_n$-orbits of $X^n$ to the $S_n$-orbits of $Y^n$, and consequently
induces a map $\varphi^{(n)}:X^{(n)}\to Y^{(n)}$, which we will call the {\em $n$-fold symmetric power of the map $\varphi$.} It is clear that this construction is 
functorial, i.e., if $\varphi:X\to Y$ and $\psi:Y\to Z$ are given maps, then $(\psi\circ\varphi)^{(n)}=\psi^{(n)}\circ \varphi^{(n)}$ as maps from $X^{(n)}$ to $Z^{(n)}$.
When $X$ and $Y$ are subsets   $\cx$, and the symmetric products $X^{(n)}$  and $Y^{(n)}$ are identified with the sets 
$\Sigma^n X=\pi(X^n)\subset\cx^n$ and 
$\Sigma^n Y=\pi(Y^n)\subset\cx^n$, we denote the induced map $\varphi^{(n)}$  (identified with  a map from $\Sigma^nX$ to $\Sigma^n Y$ )
by $\Sigma^n\varphi$.  Consequently, $\Sigma^n\varphi$ is characterized by the equation \eqref{eq-sigmanpsi}, i.e., $(\Sigma^n\varphi)(\pi(z))= \pi (\varphi^n(z))$, or equivalently by the commutative diagram:
\[\begin{CD}
X^n @>\varphi^n>> Y^n\\
@V{\pi}VV  					@V{\pi}VV\\
\Sigma^n X @>\Sigma^n\varphi>> \Sigma^nY\\
\end{CD}\]
We first note the following:
\begin{lem}\label{lem-cont} In the above diagram, $\Sigma^n\varphi$ is continuous if and only if $\varphi$ is continuous.
\end{lem}
\begin{proof}Assume that $\varphi$ is continuous, and let $\omega$ be an open subset of $\Sigma^n Y$. We need to show that $(\Sigma^n\varphi)^{-1}(\omega)$ is 
open in $\Sigma^n X$. Let $\Omega$ be an open set in $\cx^n$ such that $\omega= \Omega\cap \Sigma^n X$. Then
\begin{align*}(\Sigma^n\varphi)^{-1}(\omega)&= (\Sigma^n\varphi)^{-1}(\Omega)\cap (\Sigma^n\varphi)^{-1}(Y)\\
&=\pi\left((\varphi^n\circ \pi)^{-1}(\Omega)\right)\cap \Sigma^n X.
\end{align*}
Since $\varphi^n \circ \pi$ is continuous and $\pi$ is open by (d) of Section~\ref{sec-symmetrization}, it follows that $(\Sigma^n\varphi)^{-1}(\omega)$ is 
open and therefore $ \Sigma^n\varphi$ is continuous.

Now suppose $ \Sigma^n\varphi$ is continuous.  Let $\psi$ be its restriction to $\pi({\mathsf V}_1\cap X^n)$.  The image of $\psi$ is contained in 
$\pi({\mathsf V}_1\cap Y^n)$. By Corollary~\ref{cor-v1}
the map $\pi$ induces a biholomorphism from ${\mathsf V}_1$ to $\pi({\mathsf V}_1)$, there we have 
\[ \varphi^n|_{{\mathsf V}_1\cap X^n}= \pi^{-1}\circ \psi\circ\left(\pi|_{{\mathsf V}_1\cap X^n}\right),\]
where $\pi^{-1}$ is the biholomolomorphic inverse of $\pi$ from $\pi({\mathsf V}_1)$ to ${\mathsf V}_1$. Since $\psi$ is continuous, so is 
$\varphi^n|_{{\mathsf V}_1\cap X^n}$. But this last map is given as $(z,\dots,z)\mapsto (\varphi(z),\dots,\varphi(z))$ which is continuous if and only if $\varphi$
is continuous.\end{proof}
\begin{prop}\label{prop-resrep} There is a polynomial automorphism ${\mathfrak{P}}$ of $\cx^n$ with the following property. 
If $U\subset \cx$ is a domain with smooth
boundary, and $\varphi:U\to V$ is a holomorphic map into a bounded domain $V\subset \cx$, then
\begin{equation}\label{eq-ppsi} \Sigma^n \varphi = {\mathfrak{P}}\circ \Psi,\end{equation}
where $\Psi:\Sigma^n U\to \cx^n$ is the map whose $m$-th component function is given,  for $m\in I_n$, by
\begin{equation}\label{eq-psim-rep}
\Psi_m(s)= \frac{1}{2\pi i}\int_{\bd{U}}
(\varphi(t))^m\frac{q'(s;t)}{q(s;t)} dt,
\end{equation}
where $s\in \Sigma^n U$, $q(s;t)$ is as in \eqref{eq-q} and $q'(s;t)$ denotes the partial derivative $\dfrac{\partial q}{\partial t}(s;t)$ given by 
\eqref{eq-qprime}.
\end{prop}
\begin{proof}
Define the {\em power-sum map} $\tau:\cx^n\to \cx^n$, 
where for $j\in I_n$, the component $\tau_j$ is the $j$-th {\em power-sum polynomial}:
\[ \tau_j(z) = \sum_{\ell=1}^n (z_\ell)^j.\]
We claim that there exists a polynomial automorphism $\mathfrak{P}:\cx^n\to \cx^n$ such that
\begin{equation}\label{eq-ptau} \pi=\mathfrak{P}\circ\tau.\end{equation}
It is known from classical algebra (cf. \cite[pp. 312 ff.]{kurosh})  that the either of the collections $\{\tau_j\}_{j=1}^n$  or $\{\pi_j\}_{j=1}^n$ 
freely generates the ring of symmetric polynomials  in $n$ variables. 
It follows that for $ j \in I_n$, there are polynomials ${\mathfrak P}_j, {\mathfrak Q}_j$ in $n$ variables such that
$ \pi_j = {\mathfrak P}_j(\tau_1, \dots, \tau_n)$
and
$ \tau_j = {\mathfrak Q}_j(\pi_1,\dots,\pi_n).$
It is possible to write down the polynomials ${\mathfrak P}_j, {\mathfrak Q}_j$
explicitly using the Newton identities for symmetric polynomials  (cf. \cite[p. 323]{kurosh}.)
Define polynomial self-maps $\mathfrak{P}$ and $\mathfrak{Q}$ of $\cx^n$   by setting ${\mathfrak P}=({\mathfrak P}_1,\dots, {\mathfrak P}_n)$
and ${\mathfrak Q}=({\mathfrak Q}_1,\dots,{\mathfrak Q}_n)$.
Then  \eqref{eq-ptau} holds and ${\mathfrak P}\circ {\mathfrak Q}$
is the identity on $\cx^n$, and therefore ${\mathfrak P}$ is a polynomial automorphism of $\cx^n$. 

Recall that ${\mathsf V}_n\subset\cx^n$ is the set of points whose coordinates are $n$ distinct complex numbers,  and assume that $s$ is a point in 
$\pi({\mathsf V}_n)\cap \Sigma^n U$, i.e., the roots of $q(s;t)=0$ are $n$ distinct points $z_1,\dots, z_n$ in $U$. Then the integrand in \eqref{eq-psim-rep}
is a meromorphic function  on $U$ with simple poles at the points $t=z_1,\dots, t=z_n$. We can evaluate the integral using the residue theorem, which leads 
to the representation
\begin{align*} \Psi_m(s)& = \sum_{j=1}^n \left(\varphi(z_j)\right)^m\\
&= \tau_m(\varphi(z_1),\dots,\varphi(z_n))\\
&=\tau_m(\varphi^n(z)),
\end{align*}
so that we have
\[ \Psi(s) =\tau(\varphi(z_1),\dots, \varphi(z_n))= \tau(\varphi^n(z)).\]
Therefore, with $\mathfrak{P}$ as in \eqref{eq-ptau} we have
\begin{align*} {\mathfrak P}(\Psi(s))&= {\mathfrak P}(\tau(\varphi(z_1),\dots, \varphi(z_n))\\
&= \pi(\varphi(z_1),\dots, \varphi(z_n))\\
&= \left(\Sigma^n\varphi\right)(s),
\end{align*}
where we have used  the relation \eqref{eq-sigmanpsi} defining $\Sigma^n\varphi$. Therefore \eqref{eq-ppsi} is established if
$s\in \pi({\mathsf V}_n)\cap \Sigma^n U$.  To complete the proof we note that by Lemma~\ref{lem-cont}, $\Sigma^n\varphi$ is a continuous mapping 
from $\Sigma^n U$ to $\Sigma^n V$.
We also claim that the map $\Psi:\Sigma^n U\to \cx^n$ is holomorphic. Indeed, for any $j\in I_n$, by differentiation under the integral sign we have
\begin{equation}\label{eq-psimholo}  \frac{\partial}{\partial\overline{s}_j}\Psi_m(s) = \frac{1}{2\pi i} \int_{\bd U} (\varphi(t))^m \frac{\partial}{\partial\overline{s}_j}\left( \frac{q'(s;t)}{q(s;t)}\right)dt =0. \end{equation}
Therefore, in the equation \eqref{eq-ppsi} we see that both sides extend continuously to the complement of ${\mathsf V}_n$, and the result follows.
\end{proof}
\begin{cor}\label{cor-holo} Let $\varphi:U\to V$ be a mapping of planar domains. Then the induced map $\Sigma^n \varphi:\Sigma^n U \to \Sigma^k V$ is 
 holomorphic if and only if $\varphi$ is holomorphic.
\end{cor}
\begin{proof} First assume that $\varphi$ is holomorphic. For smoothly bounded $U$ it was proved in course of the proof of Proposition~\ref{prop-resrep} that 
$\Sigma^n\varphi$ is holomorphic. Otherwise, let $\{U_\nu\}_\nu$ be a family of bounded domains  with smooth boundary such that $U_\nu\subset U_{\nu+1}$ and $\bigcup U_\nu = U$.  Clearly
$\Sigma^n (\varphi|_{U_\nu})= (\Sigma^n \varphi)|_{\Sigma^n U_\nu}$, and by the previous proposition, we can write 
$\Sigma^n (\varphi|_{U_\nu})=\mathfrak{P}\circ \Psi$, with $\Psi$ as defined by \eqref{eq-psim-rep} with $U$ replaced by $U_\nu$. Further, by differentiation under the integral sign as in \eqref{eq-psimholo} we see that $\Psi$ and therefore $\Sigma^n (\varphi|_{U_\nu})= (\Sigma^n \varphi)|_{\Sigma^n U_\nu}$ is
holomorphic. Since this is true for each $\nu$ the result follows.

If $\Sigma^n\varphi$ is holomorphic, so is its restriction to the one-dimensional  complex submanifold $\pi({\mathsf V}_1)\cap\Sigma^n U$.
But using the biholomorphism $\pi|_{{\mathsf V}_1}$, we see that this induces a holomorphic map from ${\mathsf V}_1\cap U^n$ to
${\mathsf V}_1\cap V^n$. Using the natural coordinate $w$ on ${\mathsf V}_1$ we see that $\varphi$ is itself holomorphic.
\end{proof}

\section{A Cauchy-type integral}\label{sec-cauchy}
Let $U$ be a  bounded domain in the complex plane with rectifiable boundary, so that we have the Cauchy theory at our disposal.
For a positive integer $m$ we consider the function ${\mathsf J}u$ 
on $\Sigma^nU$ defined by
\begin{equation}\label{eq-amdef} {\mathsf J}u(s)= \int_{\bd U} \frac{u(t)}{q(s;t)^m}dt\end{equation}
where  $q$ is as in \eqref{eq-q}.
(Of course ${\mathsf J}$ depends on $m,n$ but for simplicity we suppress this from the notation.) Clearly ${\mathsf J}$ is a linear operator from $\mathcal{C}(\bd U)$ to $\mathcal{O}(\Sigma^n U)$, since the kernel $q(s;t)$ is holomorphic in the 
parameter $s\in U^n$.

For a domain $U\subset \cx$ with smooth boundary, and an integer $k\geq 0$,  let $\mathcal{A}^k(U)$  denote the space of functions 
 $\mathcal{C}^k(\overline{U})\cap \mathcal{O}(U)$, i.e. functions which are $k$ times continuously differentiable on $\overline{U}$ which  are holomorphic in $U$.  We note that $\mathcal{A}^{k}(U)$ 
is a closed subspace of the Banach space $\mathcal{C}^k(\overline{U})$, and therefore a Banach space with the norm of $\mathcal{C}^k(\overline{U})$.
Let $\mathcal{A}(\Sigma^nU)$ denote the subspace of the space  $\mathcal{C}(\overline{\Sigma^n U})$ of continuous functions on
$\overline{\Sigma^nU}$ which are holomorphic on $\Sigma^n U$.

The following result will be used in the proof of Theorems~\ref{thm-akn} and \ref{thm-dbar}.
\begin{prop}\label{prop-am} For $m\geq 1$, the operator ${\mathsf J}$ is bounded from $\mathcal{A}^{mn-1}(U)$ to $\mathcal{A}(\Sigma^nU)$.
If $u$ is a holomorphic function on $U$ which is not in $\mathcal{A}^{mn-1}(U)$, then ${\mathsf J}u$ is not
in $\mathcal{A}(\Sigma^nU)$. 
\end{prop}
The proof will be in several steps.  First, let
 $\mathcal{O}_{\rm sym}(U^n)$ be the space of holomorphic functions on $U^n$ which are symmetric in the variables $(z_1,\dots,z_n)$.
Define a linear operator from $\mathcal{C}(\bd U)$ to $\mathcal{O}_{\rm sym}(U^n)$  by setting 
\begin{equation}\label{eq-tm}
(T_nu)(z)=\int_{\bd U}\frac{u(t)}{\prod\limits_{j=1}^n(t-z_j)^m} dt.
\end{equation}
(Note that the subscript in $T_n$ refers to the dimension of the product $U^n$.)
The kernel function $\prod\limits_{j=1}^n(t-z_j)^{-m}$ is holomorphic in the variable $z\in U^n$ and symmetric in the variables $(z_1,\dots,z_n)$, 
so $T_nu\in \mathcal{O}_{\rm sym}(U^n).$   Let $H^\infty_{\rm sym}(U^n)\subset \mathcal{O}_{\rm sym}(U^n)$ be the Banach space of bounded holomorphic functions on $U^n$  symmetric with respect to the variables $(z_1,\dots,z_n)$ 
with the supremum norm. We first prove the following:
\begin{lem} \label{lem-tm}The operator $T_n$ is bounded from $\mathcal{A}^{mn-1}(U)$ to $H^\infty_{\rm sym} (U^n).$ \end{lem}
\begin{proof} Since 
$T_n u \in \mathcal{O}_{\rm sym}(U^n)$, we only need to show that $T_n u$ is a  bounded function. To do this, we use the stratification described in Section~\ref{sec-strata}.  We first show that $T_n u$ is bounded on ${\mathsf V}_1\cap U^n$, and then show inductively that if $T_n u$ is bounded on  $ \left( \cup_{j=1}^k {\mathsf V}_j\right)\cap U^n,$
then it is also bounded on  $ \left( \cup_{j=1}^{k+1} {\mathsf V}_j\right)\cap U^n.$ Applying the induction step $n$ times, we conclude that $T_n u$ is bounded on 
$U^n$.

Recall that we use $C$  or $C_j$ to denote a constant which depends on the integers $m, n$ and the domain $U$, but not on the function $u$. The actual value of $C$ or $C_j$ at various occurrences may be different.

For the base step, let $z\in {\mathsf V}_1\cap U^n$. Therefore, there exists $w\in U$ such that $z_1=\dots=z_n=w$. Hence, we get from 
\eqref{eq-tm} that 
\[
  T_nu(z)=\int_{\bd U}\frac{u(t)}{(t-w)^{nm}}dt.
\]
By the Cauchy integral formula, we obtain that 
\[
 T_nu(z)=\left.{C}\left(\frac{d}{dt}\right)^{mn-1}
 \left(u(t)\right)\right\vert_{t=w} =  C\, u^{(mn-1)}(w).
\]
Since $u\in\mathcal{C}^{mn-1}(\overline{U})$, the function $u^{(mn-1)}$  is bounded on ${U}$. Hence, 
$T_nu$ is bounded on ${\mathsf V}_1\cap U^n$.

We now use induction for the proof. Assume that $T_nu$ is bounded on  $ \left( \cup_{j=1}^k {\mathsf V}_j\right)\cap U^n$, and we want to show that it is bounded
on  $ \left( \cup_{j=1}^{k+1} {\mathsf V}_j\right)\cap U^n.$  Let $M$ be the supremum of $\abs{T_n u}$ on $ \left( \cup_{j=1}^k {\mathsf V}_j\right)\cap U^n$. 
Since $T_n u$ is continuous, it follows that there is a neighborhood  $W_k$ of 
$ \left( \cup_{j=1}^k {\mathsf V}_j\right)\cap U^n$  such that \[
 \abs{T_nu(z)}\leq 2M \;\;\text{for all}\;\; z\in W_k.
\]
Therefore, to complete the induction step, it is enough to show that $T_nu$ is 
bounded on ${\mathsf V}_{k+1}\setminus W_k$.

Let $z\in {\mathsf V}_{k+1}\cap U^n$. Then there exists a partition $Q\in \mathcal{P}_{k+1}^n$ of $I_n$ 
such that $z\in S(Q)\cap U^n$, where $Q=\{Q_1,\dots, Q_{k+1}\}$ and $S(Q)$ is as defined in 
\eqref{eq-sp}. We write $v_1,\dots, v_{k+1}$ as the distinct coordinates corresponding 
to $Q_1,\dots,Q_{k+1}$ in the partition $Q$, respectively. Therefore, again, from \eqref{eq-tm}, we get that 
\[
  T_nu(z)=\int_{\bd U}\frac{u(t)}{\prod\limits_{l=1}^{k+1} (t-v_l)^{\abs{Q_l}m}}dt.
\]
Since the integrand is a meromorphic function on $U$ extending smoothly to $\bd U$, we can use the residue formula to compute its contour integral:
\[  T_nu(z)=  \sum_{j=1}^{k+1}C_j\res_{t=v_j}h(t), \]
where
\[ h(t)= \frac{u(t)}{\prod\limits_{l=1}^{k+1} (t-v_l)^{\abs{Q_l}m}}.\]
 Using a classical formula  (cf. \cite[Proposition~4.5.6]{gk}) to compute the residue 
at the pole $v_j$ of multiplicity $\abs{Q_j}m$, we have
\begin{align*} \res_{t=v_j}h(t)&=\left.{\left(\abs{Q_j}m-1\right)!} \left( \frac{d}{dt}\right)^{\abs{Q_j}m-1}
\left((t-v_j)^{\abs{Q_j}m}\cdot h(t)\right)\right\vert_{t=v_j}\\
&=\left.{\left(\abs{Q_j}m-1\right)!} \left( \frac{d}{dt}\right)^{\abs{Q_j}m-1}
\left(\frac{u(t)}{\prod\limits_{l\in I_{k+1}\setminus\{j\} } (t-v_l)^{\abs{Q_l}m}}\right)\right\vert_{t=v_j}.
\end{align*}Therefore,
\[ T_nu(z)=\sum_{j=1}^{k+1}C_j \left(\frac{\partial}{\partial v_j}\right)^{\abs{Q_j}m-1} 
\left(\frac{u(v_j)}{\prod\limits_{l\in I_{k+1}\setminus\{j\}}(v_j-v_l)^{\abs{Q_l}m}}\right),
\]
where the terms on the right hand side are evaluated formally. Hence, we obtain
\begin{equation}
 \abs{T_nu(z)}\leq {C}\norm{u}_{\mathcal{C}^{mn-1}(\overline{U})}\cdot \sum_{j=1}^{k+1}
 \prod_{l\in I_{k+1}\setminus\{j\}}\abs{v_j-v_l}^{-(\abs{Q_l}+\abs{Q_j})m+1}.\label{eq-modtmu}
\end{equation}
If $z\in (S(Q)\cap U^n)\setminus W_k$, then $z_i=v_l$ for all $i\in Q_l$, $l=1,\dots, {k+1}$. 
Fix an $r$, $1\leq r\leq k+1$. We now claim that there exists $\epsilon>0$ such that 
\[
 \abs{v_r-v_l}\geq \epsilon \;\; \text{for all}\;\; l=1,\dots, k+1, l\neq r.
\]
Suppose no such $\epsilon$ exists. Then there exists $p$ with $ 1\leq p\leq k+1$, $p\neq r$, such that 
for every $\epsilon>0$,  there exist $z\in  (S(Q)\cap U^n)\setminus W_k$ such that
\[
 \abs{v_r-v_p}<\epsilon.
\]
Using the Bolzano-Weirstrass theorem, we obtain a $z^*\in (S(Q)\cap U^n)\setminus W_k$ such that for the corresponding $v_j^*$, we have
$v_p^*=v_r^*$, which means that $z^*\in S(Q^*)$, where $Q^*$ is a partition of length $k$ or less. But this contradicts the fact that $z^*$ is a limit point 
of a set bounded away from $\bigcup_{j=1}^k {\mathsf V}_j$. 
Therefore, there exists $\epsilon>0$ such that $\abs{v_r-v_l}\geq \epsilon \;\; \text{for all}\;\; l=1,\dots, k+1, l\neq r.$
Hence, we get from \eqref{eq-modtmu} that 
\[
 \abs{T_nu(z)}\leq {C}\norm{u}_{\mathcal{C}^{mn-1}(\overline{U})} \epsilon^{-N},\;\; \text{for all}\;\; z\in ({\mathsf V}_{k+1}\cap U^n)\setminus W_k, 
\]
where $N>0$ is an integer.

Since we already know that $T_nu$ is bounded on $W_k$, we conclude that $T_nu$ is bounded on ${\mathsf V}_{k+1}$. Hence, 
the induction step is complete. 
Therefore, proceeding finitely many steps, we get that $T_nu$ is a  bounded function on $U^n$. Hence, $T_nu\in H^\infty_{\rm sym}(U)$ 
for all $u\in \mathcal{A}^{mn-1}(U)$.

We therefore have shown that $T_n$ is algebraically a linear map from $\mathcal{A}^{mn-1}(U)$ to $H^\infty_{\rm sym} (U^n).$
To show that $T_n$ is a bounded operator from $\mathcal{A}^{mn-1}(U)$ to ${H}^\infty_{\rm sym}(U^n)$, by the closed graph theorem,
it suffices to show that $T_n$ is a closed operator.  Let 
$\{u_\nu\}\subset\mathcal{A}^{mn-1}(U)$ be a sequence such that $\{u_\nu\}$  converges to $u$ in  $\mathcal{A}^{mn-1}(U)$ 
and $T_nu_\nu$ converges  to $v$ in $ {H}^\infty_{\rm sym}(U)$ as $\nu\to \infty$. We want to show that  $v=T_nu.$ Suppose that
the smooth boundary $\bd U$ is the union of $M$ smooth closed curves, and let  them be parametrized as $\gamma_k:[0,1]\to\cx$,  where $k\in I_M$,
and each $\gamma_k$ is smooth.
From \eqref{eq-tm}, we have for each $z\in U^n$:
\begin{align*}
T_nu_\nu(z)&=\int_{\bd U}\frac{u_\nu(t)}{\prod\limits_{j=1}^n(t-z_j)^m}dt\\
&=\sum_{k=1}^M\int_0^1 \frac{u_\nu(\gamma_k(\tau)) \gamma_k\rq{}(\tau)}{\prod\limits_{j=1}^n(\gamma_k(\tau)-z_j)^m}d\tau.
\end{align*} 
Since   $\{u_\nu\}$  converges to $u$ in  $\mathcal{A}^{mn-1}(U)$, a fortiori, the sequence  $\{u_\nu\}$ converges uniformly to $u$ on $\bd U$, so the integrand in the $k$-th integral converges uniformly to 
\[ \tau \mapsto \frac{u(\gamma_k(\tau)) \gamma_k\rq{}(\tau)}{\prod\limits_{j=1}^n(\gamma_k(\tau)-z_j)^m}.\]
Therefore, taking the limit under the integral signs, we conclude that 
 for $z\in U^n$, 
\[
v(z)=\lim_{\nu\to\infty}T_nu_\nu(z)=\int_{\bd U}\frac{u(t)}{\prod_{j=1}^n(t-z_j)^m}dt=T_nu(z),
\]
which completes the proof.
\end{proof}
Lemma~\ref{lem-tm} allows us to conclude a weak version of Proposition~\ref{prop-am}:
\begin{cor}\label{cor-am} The operator ${\mathsf  J}$ is bounded from $\mathcal{A}^{mn-1}(U)$ to $H^\infty(\Sigma^nU)$.
\end{cor}
\begin{proof} Let $\pi^*:H^\infty(\Sigma^n U)\to  H^\infty_{\rm sym}(U^n)$ be defined as $\pi^*(u)= u\circ \pi$.  Then it is not difficult to see that 
$\pi^*$ is an isometric isomorphism of Banach spaces. But it is clear that $T_n = \pi^*\circ {\mathsf  J}$.  The result follows.
\end{proof}
Let $\lo(\Sigma^n U)$ denote the space of Lipschitz holomorphic functions on $\Sigma^n U$. Recall that this is a Banach space with the norm
\[ \norm{u}_{\lo} = \norm{u}_{\sup} +\sup_{s,s\rq{}\in \Sigma^n U} \frac{\abs{u(s)-u(s\rq{})}}{\abs{s-s\rq{}}}. \]
Clearly, $\lo(\Sigma^n U)\subset\mathcal{A}(\Sigma^n U)$ with continuous inclusion.
We can now prove the following:
\begin{lem}\label{lem-lip}The operator ${\mathsf J}$ is bounded from $\mathcal{A}^{2mn-1}(U)$ to $\lo(\Sigma^n U)$.
\end{lem}
\begin{proof}
We have, for $u\in \mathcal{A}^{2mn-1}(U)$ :
\begin{align}
{\mathsf J} u(s)-{\mathsf  J}u(s')
&=\int_{\bd U}u(t)\left(\frac{1}{q(s;t)^{m}} - \frac{1}{q(s';t)^{m}}\right)dt \notag\\
&= \int_{\bd U} u(t)\frac{q(s';t)^{m}-q(s;t)^{m}}{q(s;t)^{m}q(s';t)^{m}} dt. \label{eq-qq'}
\end{align}
Factorizing the numerator of the integrand in \eqref{eq-qq'}, we have 
\begin{align}
q(s;t)^{m}-q(s';t)^{m}&=\left(q(s';t)-q(s;t)\right)\cdot\sum_{k=1}^{m-1} q(s;t)^k q(s';t)^{m-1-k}\nonumber\\
&=\left(\sum_{j=1}^n(-1)^j(s_j-s'_j)t^{n-j}\right)\cdot \left(\sum_{\ell=0}^{n(m-1)}P_\ell(s,s')t^\ell\right)\nonumber\\
&= \sum_{j=1}^n\sum_{\ell=0}^{n(m-1)}(-1)^j(s_j-s'_j)P_\ell(s,s')t^{\ell+n-j},\label{eq-facnum}
\end{align}
where in the second line, we have used the expression of $q(s,t)$ and $q(s';t)$  in the first factor, and rearranged the second factor in descending powers of 
$t$, so that $P_\ell$, $\ell=0,\dots,n(m-1)$, are polynomials on $\cx^n\times \cx^n$.

Hence, in view of \eqref{eq-facnum}, \eqref{eq-qq'} reduces to 
\begin{align}
 {\mathsf J}u(s)-{\mathsf J}u(s')&
=\int_{\bd U} u(t)\frac{\sum_{j=1}^n\sum_{\ell=0}^{n(m-1)}(-1)^j(s_j-s'_j)P_\ell(s,s')t^{\ell+n-j}}{q(s;t)^{m}q(s';t)^{m}}dt\notag\\
 &=\sum_{j=1}^n(-1)^j(s_j-s'_j)\left(\sum_{\ell=0}^{n(m-1)}P_\ell(s,s')\int_{\bd U} \frac{t^{\ell+n-j} u(t)}{q(s;t)^{m}q(s';t)^{m}}dt\right). \label{eq-intqq'}
 \end{align}
We now claim that for each $j$ with $ 1\leq j \leq n$ and each $\ell$ with  $0\leq \ell \leq n(m-1)$, the functions on $\Sigma^n U\times \Sigma^nU$ defined by 
\[ H_{j,\ell}(s,s')= \int_{\bd U} \frac{t^{\ell+n-j} u(t)}{q(s;t)^{m}q(s';t)^{m}}dt\]
(the  integrals in \eqref{eq-intqq'}) are bounded. To estimate $H_{j,\ell}$, we consider the surjective map $\Pi$ from $U^{2n}$ to
 $\Sigma^n U\times \Sigma^n U$ given by $(z,z\rq{})\mapsto (\pi(z), \pi(z\rq{}))$, where $z=(z_1,\dots,z_n)$ and $z\rq{}=(z_{n+1},\dots, z_{2n})$,
where $z_j\in U$ for $j=1,\dots,2n$. To show that $H_{j,\ell}$ is bounded on $\Sigma^n U\times \Sigma^n U$, it is sufficient to show that $H_{j,\ell}\circ \Pi$ is bounded on $U^{2n}$. But for $(z,z\rq{})\in U^{2n}$ we have
\begin{align*}H_{j,\ell}\circ \Pi (z,z\rq{})&= \int_{\bd U} \frac{t^{\ell+n-j} u(t)}{\left(q(\pi(z);t)q(\pi(z');t)\right)^{m}}dt\\
&=  \int_{\bd U} \frac{t^{\ell+n-j} u(t)}{\prod_{i=1}^{2n}(t-z_i)^m}dt\\
&= T_{2n}\left( t^{\ell+n-j}u\right)
\end{align*}
By Lemma~\ref{lem-tm} the operator $T_{2n}$ is bounded from $\mathcal{A}^{2mn-1}(U)$ to $H^\infty_{\rm sym}(U^{2n})$. We therefore have
\[ \norm{H_{j,\ell}\circ \Pi }_{\sup} \leq C \norm{t^{\ell+n-j} u}_{\mathcal{C}^{2mn-1}(\overline{U})}\leq C \norm{u}_{\mathcal{C}^{2mn-1}(\overline{U})}.\]
Since the polynomial $P_\ell$ is bounded on $\Sigma^nU\times \Sigma^n U$, from \eqref{eq-intqq'}  we get the estimate
\[
\abs{{\mathsf J}u(s)-{\mathsf J}u(s')}\leq  C \norm{u}_{\mathcal{C}^{2mn-1}(\overline{U})}\abs{s-s'},
\]
which shows that ${\mathsf J}u$  is Lipschitz on $\Sigma^n U$. And we also have
\begin{align*} 
\norm{{\mathsf J}u}_{\lo (\Sigma^n U)} &= \norm{{\mathsf J}u}_{\mathcal{C}(\overline{U})}+ \sup_{s,s\rq{}\in \Sigma^n U} \frac{\abs{{\mathsf J}u(s)-{\mathsf J}u(s\rq{})}}{\abs{s-s\rq{}}}.\\
&\leq C \norm{u}_{\mathcal{C}^{mn-1}(\overline{U})}+ C \norm{u}_{\mathcal{C}^{2mn-1}(\overline{U})}\\
&\leq C\norm{u}_{\mathcal{C}^{2mn-1}(\overline{U})},
\end{align*}
which shows that the map $\mathsf{J}$ is bounded from $\mathcal{C}^{2mn-1}(\overline{U})$ to $\lo (\Sigma^n U)$.
\end{proof}

We can now complete the proof of Proposition~\ref{prop-am}:
\begin{proof}[Proof of Proposition~\ref{prop-am}] Since $\mathcal{A}(\Sigma^nU)$ is a closed subspace of the Banach space $H^\infty(\Sigma^n U)$, it is sufficient to show that for each 
$u\in \mathcal{A}^{mn-1}(U)$, we have ${\mathsf J}u \in  \mathcal{A}(\Sigma^nU)$.   Let $\mathcal{O}(\overline{U})$ denote the space of those functions $f\in \mathcal{O}(U)$ such that there is an open neighborhood
 $U_f \supset \overline{U}$ to which $f$ extends holomorphically.  It is a classical fact that  for each $k\geq 0$, the space $\mathcal{O}(\overline{U})$ is dense in $\mathcal{A}^k(U)$.  For $k=0$, this is a famous theorem of Mergelyan, and the general case may be easily reduced to the case $k=0$. Therefore 
 we can find a sequence 
$\{u_\nu\}$ in $\mathcal{O}(\overline{U})$ such that $u_\nu\to u$ in the topology of $\mathcal{A}^{mn-1}(U)$ as $\nu\to \infty$. But for each 
$u_\nu$, we have ${\mathsf J}u_\nu\in \lo(\Sigma^n U)$ by Lemma~\ref{lem-lip}. In particular,  ${\mathsf J}u_\nu \in \mathcal{A}(\Sigma^n U)$ . By continuity of 
${\mathsf J}$, we have ${\mathsf J}u_\nu \to {\mathsf J}u$ in the supremum norm in $H^\infty(\Sigma^n U)$. Since $\mathcal{A}(\Sigma^n U)$ is
closed in  $H^\infty(\Sigma^n U)$, it follows that ${\mathsf J}u \in \mathcal{A}(\Sigma^n U)$.

Now let $u\in\mathcal{O}(U)$, be such that $u\not\in \mathcal{A}^{mn-1}(U)$.
By Corollary~\ref{cor-v1}, the map $\pi$ is a biholomorphism when restricted to the 
complex line ${\mathsf V}_1= \{w v | w\in \cx\}$, where $v=(1,\dots,1)$.  As in the proof of  Lemma~\ref{lem-tm},  for $w\in U$,
we can compute the 
value of the function ${\mathsf J}u(s)$ at the point $s=\pi(wv)$ on $\pi({\mathsf V}_1)$ using the Cauchy integral formula:
\begin{align*}{\mathsf J}u(s)&= \int_{\bd U}\frac{u(t)}{(t-w)^{mn}}dt\\
&= C u^{(mn-1)}(w).
\end{align*}
Since $u^{(mn-1)}$ does not extend continuously to $\overline{U}$, it follows that ${\mathsf J}(u)$ does not extend continuously to  $\pi({\mathsf V}_1)\cap 
\Sigma^n U$. It follows that ${\mathsf J}u \not\in \mathcal{A}(\Sigma^n U)$.
\end{proof}

\section{Proof of Theorem~\ref{thm-akn}}\label{S:phi}

When $k=0$, applying Lemma~\ref{lem-cont} to the continuous map $\varphi: \overline{U}\to \overline{V}$, we see that $\Sigma^n \varphi$ is continuous 
from $\Sigma^n \overline{U}$ to $\Sigma^n \overline{V}$. Further by Corollary~\ref{cor-holo}, the map $\Sigma^n \varphi$ is holomorphic from
$\Sigma^n {U}$ to $\Sigma^n {V}$, which proves Theorem~\ref{thm-akn} when $k=0$.

Therefore, we may assume $k\geq 1$.
In view of \eqref{eq-ppsi}, the  regularity of the map $\Sigma^n\varphi$ is same as that of the mapping $\Psi$. 
We therefore want to compute the first $k$  derivatives of the function $\Psi_m$  and show that they extend continuously to the boundary of $\Sigma^n U$.  In view of the representation \eqref{eq-psim-rep} of the components $\Psi_m$ of $\Psi$,  the conclusion of Theorem~\ref{thm-akn} follows from the following 
lemma:
\begin{lem}\label{L:delG}
For $k\geq 1$,  and a function $g\in \mathcal{A}^1(U)$, let  the function $G$ be defined by 
 \begin{equation}\label{eq-Gg} G(s)= \int_{\bd{U}}
g(t)\frac{q'(s;t)}{q(s;t)} dt,
\end{equation} then $G\in \mathcal{A}^k(\Sigma^n U)$ if and only if $g\in\mathcal{A}^{kn}(U)$.
\end{lem}

\begin{proof}
By the Whitney extension theorem it is sufficient to show that for each  multi-index $\alpha$,  with $\abs{\alpha}\leq k$, the derivative
$\partial^\alpha G(s)$ is in $\mathcal{A}(\Sigma^n U)$. Again, in this proof we let $C$ denote a constant which only depends on the multi-index $\alpha$. The value of $C$ in different 
occurrences may be different. If $\alpha=(\alpha_1,\dots,\alpha_n)$ 
is a multi-index, denote by $\partial^\alpha$ the complex partial derivative
\[ \left(\dfrac{\partial}{\partial s_1}\right)^{\alpha_1}\left(\dfrac{\partial}{\partial s_2}\right)^{\alpha_2}\dots 
\left(\dfrac{\partial}{\partial s_n}\right)^{\alpha_n},\]
 acting on holomorphic functions of $n$ complex variables $(s_1,\dots,s_n)$.
We now compute $ \partial^\alpha G$. If $t\in\bd U$, the function $\dfrac{q'(s;t)}{q(s;t)}$ is holomorphic  in $s\in\Sigma^nU$. Hence, for $1\leq j \leq n$, we can differentiate \eqref{eq-Gg} 
with respect to $s_j$ under the integral sign to get: 
\[
\dfrac{\partial}{\partial s_j}G(s_1,\dots, s_n)= 
\int_{\bd {U}}g(t)\dfrac{\partial}{\partial s_j}\left(\frac{q'(s;t)}{q(s;t)}\right)dt.
\]
After choosing a locally defined branch of the logarithm of $q$, we have
\begin{equation}\label{E:interdiff}
\dfrac{\partial}{\partial s_j}\left(\dfrac{q'(s;t)}{q(s;t)}\right) = \dfrac{\partial}{\partial s_j}\left(\dfrac{\partial}{\partial t}\log q(s;t)\right)= 
\dfrac{\partial}{\partial t}\left(\dfrac{\partial}{\partial s_j}\log q(s;t)\right)= \dfrac{\partial}{\partial t}\left(\dfrac{q_j(s;t)}{q(s;t)}\right),
\end{equation}
where \begin{equation}\label{eq-qj} q_j(s;t)=\frac{\partial q(s;t)}{\partial s_j} =   (-1)^jt^{n-j}, \end{equation}
since $q(s;t)=t^n-s_1t^{n-1}+\dots +(-1)^ns_n$.
Since the first and last expressions in \eqref{E:interdiff} are globally defined rational functions, they are equal for all $s\in \Sigma^n U$ and  all $t\in U$.
Therefore,
\begin{align}
\dfrac{\partial}{\partial s_j}G(s_1,\dots, s_n)&= \int_{\bd {U}}g(t)
\dfrac{\partial}{\partial t}\left(\frac{q_j(s;t)}{q(s;t)}\right)dt\nonumber\\
&= \int_{\bd {U}}\left\lbrace\frac{\partial}{\partial t}\left(g(t)\frac{q_j(s;t)}{q(s;t)}\right) - g'(t) \frac{q_j(s;t)}{q(s;t)}\right\rbrace dt\nonumber\\
&= -\int_{\bd {U}}g'(t) \frac{q_j(s;t)}{q(s;t)} dt\nonumber\\
&= {(-1)^{n-j+1}}\int_{\bd U}\frac{t^{n-j}g'(t)}{q(s;t)} dt\label{eq-gder}
\end{align}
since the integral of the first term in the second line vanishes. 
By repeated differentiation, and using \eqref{eq-qj} we have for each non-zero multi-index $\beta=(\beta_1,\dots,\beta_n)$ that
\[ \partial^\beta \left( \frac{1}{q(s;t)}\right) = C_\beta \cdot\frac{t^{\sum_{k=1}^n (n-k)\beta_k }}{q(s;t)^{\abs{\beta}+1}},\]
where $C_\beta$ is a constant which depends only on the multi-index $\beta$.

For a multi-index $\alpha\neq 0$, there exists $j$ with $ 1\leq j\leq n$, such that $\alpha_j>0$. 
Let $\beta=\alpha - e^j$, where $e^j$ is a unit multi-index  with $1$ in the $j$-th position and 0 everywhere else.
Again, differentiating under the integral sign in the equation \eqref{eq-gder},   we have:
\begin{align}
\partial^\alpha G(s)  &={(-1)^{n-j+1}} \int_{\bd U}{t^{n-j}g'(t)}\cdot\partial^\beta\left(\frac{1}{q(s;t)}\right) dt\nonumber\\
&= {(-1)^{n-j+1}}\cdot C_\beta\int_{\bd U}t^{n-j+\sum_{k=1}^n (n-k)\beta_k } g'(t) 
\cdot \frac{1}{q(s;t)^{\abs{\beta}+1}} dt\nonumber\\
&=C\int_{\bd U}g'(t)\dfrac{t^\gamma}{q(s;t)^{\abs{\alpha}}}dt,\label{eq-derG}
\end{align}
where $C$  is a constant depending on the multi-index $\alpha$ alone  and $\gamma$ is given by
\begin{equation}\label{eq-gamma} \gamma = \sum_{\ell=1}^n (n-\ell)\alpha_\ell.\end{equation}
Then, in the notation of Proposition~\ref{prop-am}, we can write, with $m=\abs{\alpha}$, that:
\[ \partial^\alpha G(s) = {\mathsf  J}(t^\gamma g\rq{}).\]
It now follows from Proposition~\ref{prop-am} that $\partial^\alpha G \in \mathcal{A}(\Sigma^n U)$ for $\abs{\alpha}\leq k$ if and only if
$t^\gamma g'$ is of class $\mathcal{A}^{kn-1}(U)$. Applying the  Whitney Extension theorem, this is equivalent to $G$ being of class $\mathcal{A}^{k}(U)$, and note that 
this also holds if $k=\infty$.
\end{proof}

\section{Proof of Theorem~\ref{thm-dbar}}
\label{sec-thm-dbar}
In the proof  of Theorem~\ref{thm-dbar} we use the following result  \cite{duf,cc}, which can be proved either by $L^2$ methods, or by integral
representations.
\begin{result}[Dufresnoy \cite{duf}; also see \cite{cc}]\label{thm-duf}
Let $\Gamma$ be a compact set in $\cx^n$, and suppose that there is a family of open pseudoconvex neighborhoods 
$\{\Omega_\nu\}_{\nu=1}^\infty$ of $\Gamma$ such that
\begin{enumerate}
\item There is a $\theta>0$ such that for $0<\epsilon<\frac{1}{2}$, there is a $\nu$ such that
\[ \Gamma(\epsilon^\theta)\subset \Omega_\nu\subset \Gamma(\epsilon)\]
where $\Gamma(\delta)$ denotes the $\delta$-neighborhood
\[ \Gamma(\delta)= \{z\in \cx^n \mid\dist(z,\Gamma)<\delta\}.\]
\item For each $\mu$ there is a $\nu_0$ such that if $\nu>\nu_0$, the set $\overline{\Omega_\nu}$ is holomorphically convex in $\Omega_\mu$.
\end{enumerate}
Then given a form $g\in \mathcal{C}^\infty_{p,q}(\Gamma)$ with $\dbar g=0$, there is 
a $u\in \mathcal{C}^\infty_{p,q-1}(\Gamma)$ such that $\dbar u=g$.
\end{result}
Note that  by condition (2),  it follows that we have
\[ \bigcap_{\nu=1}^\infty \Omega_\nu = \Gamma.\]
By definition,  a $\mathcal{C}^\infty$-function on the set $\Gamma$ is a function admitting a $\mathcal{C}^\infty$ extension to a neighborhood.

\subsection{Model case}
Let $D(z,r)$ denote the open disc in the plane with centre at $z\in\cx$ and radius $r>0$, and let $\mathbb{D}= D(0,1)$ be the open unit disc.
A domain $V$ in the plane is said to be a {\em domain with circular boundary} if there is a nonnegative integer $M$ such that $V$ can be represented as
\begin{equation}\label{eq-vrep} V = \mathbb{D} \setminus \left( \bigcup_{j=1}^M \overline{D(z_j,r_j)}\right),\end{equation}
where the $M$ closed discs $ \{\overline{D(z_j,r_j)}\}_{j=1}^M$ are pairwise disjoint and are all contained in the open unit disc $\mathbb{D}$. 
The domain $V$ is a {\em semi-algebraic} subset of $\cx=\rl^2$ (see \cite[Chapter 2]{bcr} for more information on  semi-algebraic subsets of $\rl^n$, which formally correspond to affine algebraic varieties in $\cx^n$.)
 ``Semi-algebraic"  means that the set $V$ is defined in $\rl^2$ by a finite system of polynomial equations 
and inequalities. Setting $z_j=x_j+iy_j$, it is easy to write down the  $M+1$ polynomial inequalities that define $V$: the inequality $x^2+y^2<1$ along with the 
$M$ inequalities $(x-x_j)^2+(y-y_j)^2>r_j^2$, where $j\in I_M$. A mapping between semi-algebraic sets is  said to be semi-algebraic if its 
graph is semi-algebraic.
We will need the following  facts (see \cite{bcr}):
\begin{enumerate}[(a)]
\item The class of semialgebraic  subsets is closed under various natural operations. Interiors, closures, projections of semi-algebraic sets are semi-algebraic.
The image of a semi-algebraic subset under a semi-algebraic mapping is semi-algebraic. In particular, if $V$ is any 
semi-algebraic subset of $\cx=\rl^2$, we see that the symmetric product $\Sigma^n V\subset\cx^n =\rl^{2n}$ is semi-algebraic for each $n$. This follows
since $V^n$ is clearly semi-algebraic, and the mapping $\pi:\rl^{2n}\to \rl^{2n}$ is polynomial, and in particular semi-algebraic, so $\Sigma^n V=\pi(V^n)$
is semi-algebraic. Similarly, $\overline{\Sigma^n V}=\Sigma^n(\overline{V})=\pi(\overline{V}^n)$ is semi-algebraic.
\item A {\em semi-algebraic function  } is defined to be a real valued semi-algebraic mapping. The semi-algebraic functions on a semi-algebraic set are closed under the standard
operations of taking sums, products, negatives etc, as well as under taking maxima or minima of finitely many such functions.
\item If $A$ is a non-empty semi-algebraic subset of $\rl^N$ then  $x\mapsto {\rm dist}(x,A)$ is  a continuous semi-algebraic function on $\rl^n$, 
which vanishes precisely on the closure of $A$.
\item ({\L}ojasiewicz inequality.) Let $B$ be a compact semi-algebraic subset of $\rl^N$ and let $f,g:B\to \rl$ be two continuous semi-algebraic 
functions such that $f^{-1}(0)\subset g^{-1}(0)$. Then there is an integer $P>0$ and a constant $C\in \rl$ such that $\abs{g}^P\leq C \abs{f}$ on $B$.
\end{enumerate}
We can now prove a special case of Theorem~\ref{thm-dbar}:
\begin{prop}\label{prop-dbarcirbdy}
If $V$ is a domain with circular boundary, then on $\Sigma^nV$ the $\dbar$-problem is globally regular.
\end{prop}

\begin{proof} We will verify that $\Gamma=\overline{\Sigma^n V}$ satisfies the hypotheses of  Result~\ref{thm-duf}
quoted at the beginning of this section.

Representing $V$ as in \eqref{eq-vrep},  
 consider the function $\psi$ defined in a neighborhood of $\overline{V}$
 by setting
 \begin{equation}\label{eq-shpsi}  \psi(z) = \max\left( \abs{z}, \max_{j=1,\dots, M}\frac{r_j}{\abs{z-z_j}} \right)-1 \end{equation}
Observe that $\psi<0$ is precisely the set $V$, and in a neighborhood of $\overline{V}$ the function $\psi$ is continuous, subharmonic  and  semi-algebraic. Define
$V_\nu= \left\{z\in\cx: \psi(z)<\frac{1}{\nu}\right\}$, and set
$
\Omega_\nu=\Sigma^n \left(V_\nu\right)=\pi\left((V_\nu)^n\right),
$
so that $\{\Omega_\nu\}$ is a system of pseudoconvex neighborhoods of $\Gamma$.  We need to verify the two conditions of Result~\ref{thm-duf}.

Define real valued functions $f,g$  as follows. Let $s$ be in a neighborhood of $\Gamma$, and let $z_1,\dots, z_n$ be the roots of the equation
$q(s;t)=0$, with $q$ as in \eqref{eq-q},  so that $\pi(z_1,\dots,z_n)=s$. Define $\displaystyle{ f(s) = \max_{j\in I_n} \psi(z_j,V).}$
It is straightforward to verify that this is a well-defined continuous
semi-algebraic plurisubharmonic  function on a  neighborhood $B$  of $\overline{\Sigma^n V}$.  After shrinking, the neighborhood  $B$ is compact and
semi-algebraic.  We let $g$ be the function on $B$ given by $ g(s) = \displaystyle{\dist(s,\Sigma^n V)},$
where on the right hand side the distance in $\cx^n$ is meant.  Observe that $f^{-1}(0)=g^{-1}(0)=\overline{\Sigma^n V}$,  the level set
$\{s\in B\colon f(s)<\frac{1}{\nu}\}$ is precisely the set $\Omega_\nu$ and the  level set $\{s\in B\colon g(s)<\epsilon\}$ is the $\epsilon$-neighborhood
$\Gamma(\epsilon)$ of $\Gamma=\overline{\Sigma^n V}$.

   Therefore, applying 
the {\L}ojasiewicz inequality twice we see that for $s\in B$, we have
\begin{equation}\label{eq-loja} C_1 f(s)^{N_1} \leq g(s) \leq  C_2 f(s)^{\frac{1}{N_2}},\end{equation}
where $N_1,N_2$ are positive integers and $C_1,C_2$ are positive constants. Given small $\epsilon>0$, choose the integer $\nu$ so  that
\[ \frac{1}{2} \left(\frac{\epsilon}{C_2}\right)^{N_2} <\frac{1}{\nu} <\left(\frac{\epsilon}{C_2}\right)^{N_2}.\]
Such $\nu$ exists for small $\epsilon$. Also, let $\theta$ be so large that for all small enough $\eta$, we have
\[ \frac{1}{C_1^{\frac{1}{N_1}}}\eta^{\frac{\theta}{N_1}}<  \frac{1}{2C_2^{N_2}}\eta^{N_2}.\]

Let $s\in \Omega_\nu$ so that $f(s)<\frac{1}{\nu}$. Then we have by the right half of \eqref{eq-loja}:
\[ g(s) < C_2 f(s)^{\frac{1}{N_2}}< \epsilon .\]
Also if $s\in \Gamma(\epsilon^\theta)$, i.e., $g(s)<\epsilon^\theta$, we have $ \displaystyle{f(s) < \frac{1}{\nu}},$
so that $\Gamma(\epsilon^\theta)\subset \Omega_\nu\subset \Gamma(\epsilon)$. This verifies the first condition.

Let $\nu>\mu$. Note that the domain $\Omega_\mu$ is identical to  $\{f<\frac{1}{\mu}\}$ and the subdomain $\Omega_\nu$ is similarly given by $\{f<\frac{1}{\nu}\}$.
Since $f$ is plurisubharmonic, it follows that the plurisubharmonic hull of $\overline{\Omega}_\nu$ in $\Omega_\mu$ is itself, i.e., $\overline{\Omega}_\nu$ is 
plurisubharmonically convex in $\Omega_\mu$.  Since $\Omega_\mu$ is pseudoconvex, the notions of plurisubharmonic convexity and holomorphic
convexity coincide on $\Omega_\mu$ (see \cite[Theorem~4.3.4]{hor}.) This establishes the second condition of Result~\ref{thm-duf}, and therefore the the proposition. 
\end{proof}

{\em Remarks:} (1) In fact the closed symmetrized polydisc $\Sigma^n\overline{\mathbb{D}}$ is polynomially convex in 
$\cx^n$ (see \cite[Theorem~1.6.24]{stout}.)\\

(2) The method of Proposition~\ref{prop-dbarcirbdy} may also be used to show that the $\dbar$-problem is globally regular on $\Sigma^n V$, when $V$ 
is a  {\em subanalytic} domain in $\cx$ (For definition and basic properties of subanalytic sets, see \cite{bimil}.) We need to replace
 the subharmonic function  $\psi$ of \eqref{eq-shpsi}  by the  function
 $\widetilde{\psi}(z)=\dist(z,V)$, where the usual distance in $\cx$ is meant. The function $\widetilde{\psi}$ is not subharmonic in general, but
if we define $f$ as in the proof above setting $f(s) =\max_{j\in I_n} \widetilde{\psi}(z_j,V)$,  it is not difficult to see that the sublevel sets are still pseudoconvex, and an argument parallel to the one given above, using a subanalytic version of the {\L}ojasiewicz inequality completes the proof.
\subsection{Proof of Theorem~\ref{thm-dbar}}
 For $q\geq 1$, let $f\in \mathcal{C}^\infty_{p,q}(\overline{\Sigma^nU})$ be a $(p,q)$-form with $\dbar f=0$. We need to show that the
$\dbar$-problem
\begin{equation}\label{eq-dbaru}
\dbar u=f \quad \text{such that}\quad u\in\mathcal{C}^\infty_{p,q-1}(\overline{\Sigma^nU})
\end{equation}
has a solution $u$  smooth up to the boundary. By well-known classical results in complex analysis of one variable (cf. \cite[Theorem~7.9]{conway} and 
\cite[Theorem IX.35, p. 424]{tsuji}),
since $U\Subset\cx$ with $\mathcal{C}^\infty$-smooth boundary,  there exists a biholomorphic map  $\varphi$ from a domain $V$ with 
circular boundary to $U$ which  extends to a $\mathcal{C}^\infty$ diffeomorphism from $\overline{V}$ to $\overline{U}$. 
From Theorem~\ref{thm-akn}, it follows that
$\Sigma^n\varphi:\Sigma^nV\to\Sigma^nU$ extends as a  $\mathcal{C}^\infty$ map  up to $\overline{\Sigma^nV}$. 
Since $\varphi$ is a biholomorphism,  the same argument applied to  $\varphi^{-1}$ shows that $\Sigma^n\varphi^{-1}:\Sigma^nU\to\Sigma^nV$ is holomorphic map that extends smoothly 
from $\overline{\Sigma^nU}$ to $\overline{\Sigma^nV}$. We observe that 
$
(\Sigma^n\varphi)^{-1}=\Sigma^n\varphi^{-1}
$
on  $\overline{\Sigma^nU}$, so that $\Sigma^n \varphi$ is a biholomorphism.

The pull back $(\Sigma^n\varphi)^*f$ of $f$ by $\Sigma^n\varphi$,  is therefore a form  in $\mathcal{C}^\infty_{p,q}(\overline{\Sigma^nV})$. 
Denoting  $g=(\Sigma^n\varphi)^*f$,  and noting that $\dbar g=\dbar (\Sigma^n\varphi)^*f=(\Sigma^n\varphi)^*(\dbar f)=0$,
  we consider the $\dbar$-problem $\dbar v=g$ on $\Sigma^n V$. 
By Proposition~\ref{prop-dbarcirbdy} we obtain that the $\dbar$-problem is globally regular on $\Sigma^nV$, and therefore,
there exists a $v\in\mathcal{C}^\infty_{p,q-1}(\overline{\Sigma^nV})$ such that $\dbar v=g$.  We now let 
$u=(\Sigma^n\varphi^{-1})^*v.$  This is a solution of \eqref{eq-dbaru} that is in $\mathcal{C}^\infty_{p,q-1}(\overline{\Sigma^nU})$.
This completes the proof of Theorem~\ref{thm-dbar}.
\subsection{Non-Lipschitz nature of the boundary of a symmetric product}
As a consequence of Theorem~\ref{thm-akn},  and the use of the model \eqref{eq-vrep} 
we can show that the boundary of $\Sigma^n U$ is not Lipschitz:
\begin{prop}\label{prop-unonlip}Let $U\Subset\cx$ be a domain with $\mathcal{C}^\infty$ boundary. For $n\geq 2$,  the symmetric product
$\Sigma^n U$ does not have Lipschitz boundary.\end{prop}
We begin with a special case of this proposition, which was shown to us by by {\L}ukasz Kosi\'{n}ski:
\begin{lem}[\cite{kos}]\label{lem-kos} Proposition~\ref{prop-unonlip}  holds for $\Sigma^2 \mathbb{D}$. In fact near the point $(2,1)\in \bd\Sigma^2\mathbb{D}$ the 
boundary fails to be Lipschitz.\end{lem}
\begin{proof} It is well-known (see \cite{gamma}) that we can represent the domain  $\Sigma^2\mathbb{D}$ as
\begin{equation}\label{eq-symbidisc}
\left\{(s_1,s_2)\in \cx^2\colon\abs{s_1-\overline{s_1}s_2} + \abs{s_2}^2 <1\right\}.
\end{equation}
We will show that there is no four dimensional cone with vertex at $(2,1)\in \bd \Sigma^2\mathbb{D}$ inside the symmetrized bidisc.
Let $\ell$ be a parametrized line segment contained in the domain $\Sigma^2\mathbb{D}$ passing through the point $(2,1)$, 
given by $\ell(t)=(2-bt, 1-at) ,$
where $t\in (0,r)$, $r>0$, and $a,b\in \cx$. In view of \eqref{eq-symbidisc}, we get that 
\[
\abs{2-bt-(2-\overline{b}t)(1-at)}+\abs{1-at}^2<1.
\]
Computing the above expression we have 
\[
\abs{-2it\Im b+2at-a\overline{b}t^2}-2t\Re a+\abs{a}^2t^2<0.
\]
Dividing the above expression by $t$ and then taking limit as $t\to 0^+$ we get $
\abs{a-i\Im b}\leq \Re a, $
which is possible only when $\Im a=\Im b$. This forces any  cone inside $\Sigma^2\mathbb{D}$ with vertex at $(2,1)$ 
to be at most of dimension three, and consequently, the domain cannot be Lipschitz near $(2,1)$.
\end{proof}

Now let $V$ be a domain with circular boundary, represented as in \eqref{eq-vrep}. We can strengthen the above lemma in the following way:

\begin{lem} For $n\geq 2$, the domain  $\Sigma^n{V}$ is not Lipschitz.
\end{lem}
\begin{proof}Consider the polynomial 
map $\Phi$ from $\cx^2\times \cx^{n-2}$ to $\cx^n$ given in the following way. For $s\in \cx^2$,
let $(z_1,z_2)\in\cx^2$ be such that $\pi^2(z_1,z_2)=s$, where $\pi^2:\cx^2\to \cx^2$ is the symmetrization map. Similarly for $\sigma\in\cx^{n-2}$, let 
$(z_3,\dots,z_n)\in\cx^{n-2}$ be such that $\pi^{n-2}(z_3,\dots, z_n)=\sigma$, where again $\pi^{n-2}:\cx^{n-2}\to \cx^{n-2}$ is the symmetrization
map on $\cx^{n-2}$. We define $\Phi(s,\sigma)=\pi^n(z_1,\dots, z_n)$, where $\pi^n$ is the symmetrization map on $\cx^n$. A calculation shows
that if we set $\sigma_0=1$, and $\sigma_{-1}=\sigma_{n-1}=\sigma_n=0$, the components of the map $\Phi=(\Phi_1,\dots,\Phi_n)$ are given by
\[ \Phi_k(s,\sigma)= s_2 \sigma_{k-2}+s_1\sigma_{k-1}+\sigma_k.\]
Taking derivatives, we see that the Jacobian matrix of $\Phi$ is given by the $n\times n$ matrix
\[ \Phi\rq{}(s,\sigma)
= \begin{pmatrix}
\sigma_0 & \sigma_{-1}	& 1 & 0 & \dots&\dots&0\\
\sigma_1&\sigma_0  	      &s_1 & 1 &\dots&\dots&0\\
\sigma_2&\sigma_1		      &s_2& s_1&1&\dots&0\\
\vdots&\vdots&\vdots&\vdots&\vdots&\vdots&\vdots\\
\sigma_{n-3}&\sigma_{n-2}& 0&  \dots &s_2&s_1&1\\
\sigma_{n-2}&\sigma_{n-3}& 0&  \dots&\dots &s_2&s_1\\
\sigma_{n-1}&\sigma_{n-2}& 0&  0&\dots &0&s_2\\
\end{pmatrix}
\]

Without loss of generality, we can assume that $0\in V$. If this is not already the case, we can apply a holomorphic automorphism
of the unit disc mapping some point $p\in V$ to 0, and replace $V$ by the image under this automorphism which is still a domain with circular boundary.

 Note that then the origin of $\cx^{n-2}$ is an interior point of the domain $\Sigma^{n-2}V$.
It is clear that $\Phi(\Sigma^2 V\times\Sigma^{n-2}V)= \Sigma^n V$.

Let $P=(2,1)\times 0 \in  \cx^2\times\cx^{n-2}$ . Since $(2,1)\in \bd \Sigma^2 V$ and $ 0\in \Sigma^{n-2}V$, it is clear that  
$P\in \bd (\Sigma^2 V\times\Sigma^{n-2}V)$, and therefore  $\Phi(P)\in
\bd \Sigma^n V$.  By the previous lemma, $P$ is a non-Lipschitz point of the boundary of $\Sigma^2 V\times\Sigma^{n-2}V$.
Therefore, it suffices to show that at the point $P$
the map $\Phi$ is a local biholomorphism.

Noting that we have $\sigma_0=1$, and $\sigma_{-1}=\sigma_{n-1}=\sigma_n=0$, and substituting the coordinates of $P$ as 
 $\sigma=(\sigma_1,\dots,\sigma_{n-2})=0$, and $(s_1,s_2)=(2,1)$, we see that
$\Phi\rq{}(P)=\Phi\rq((2,1),0)$ is an upper triangular matrix with main diagonal 
$(1,1, 1,\dots,1)$. Therefore we have that  $\det \Phi\rq{}(P)=1$, and therefore
$\Phi$ is a local biholomorphism in a neighborhood of $P$.
\end{proof}

\begin{proof}[Proof of Proposition~\ref{prop-unonlip}] We proceed as in the proof of Theorem~\ref{thm-dbar}. Let $V$ be a domain with circular boundary such that there is a biholomorphic
map $\varphi:V\to U$ which extends as a $\mathcal{C}^\infty$ diffeomorphism from $\overline{V}$ to $\overline{U}$. Then by 
Theorem~\ref{thm-akn}, we have an induced diffeomorphism $\Sigma^n \varphi:\Sigma^n \overline{V} \to \Sigma^n\overline{U}$, which is 
biholomorphic in the interior. Consequently we obtain that  the boundary of $\Sigma^n U$ is diffeomorphic to  the
boundary of $\Sigma^n V$, and the result follows since the boundary of $\Sigma^n V$ is not Lipschitz by the previous lemma.
\end{proof}

\section{Proof of Theorem~\ref{thm-proper}}
We begin with the following lemma:
\begin{lem}\label{lem-bdy} Let $m\geq 1$ and let $\Omega\subset \cx^m$ and $V\subset\cx$ be domains, and let 
$\varphi:\Omega \to\cx^n$ be a holomorphic mapping whose 
image lies in $\bd (\Sigma^n V)$. Then there is a point $c\in \bd V$ such that
for all $z\in \Omega$ we have $q(\varphi(z);c)=0$, where $q$ is as in \eqref{eq-q}.
\end{lem}
\begin{proof}We use the notation of Section~\ref{sec-strata}. Since $\cx^n$ is the disjoint union of the sets $\pi({\mathsf V}_k)$, where $k\in I_n$, it follows
that at least one of the sets $\varphi^{-1}(\pi({\mathsf V}_k))\subset\Omega$ has an interior. Therefore there is a $k\in I_n$ and  a  $p$ in $\Omega$ such
 that there is a neighborhood of $p$ which is mapped by $\varphi$ into $\pi({\mathsf V}_k)$. Since $\pi$ is a local biholomorphism from 
 ${\mathsf V}_k$ onto $\pi({\mathsf V}_k)$, we see that there is a holomorphic map $\psi$ from a neighborhood of $\varphi(p)$ in $\pi({\mathsf V}_k)$
 to ${\mathsf V}_k$ such that $\pi\circ \psi$ is the identity in a neighborhood of $\varphi(p)$ in $\pi({\mathsf V}_k)$. Let $\Phi=\psi\circ \varphi$.
 Then $\Phi$ is a map from a neighborhood $\omega$ of $p$ in $\cx^m$ such that $\pi\circ \Phi=\varphi$ on $\omega$. It follows that the image of $\Phi$
 lies in $\bd V^n$.
 For $j\in I_n$, denote by $\omega_j$ the set of points $\Phi^{-1}_j(\bd V)$, i.e., the points in $\omega$ mapped by the $j$-th component of $\Phi$ to $\bd V$.
 Since the image of $\Phi$ lies in $\bd V^n$, it follows that $\bigcup\limits_{j=1}^n \omega_j=\omega$. Therefore, there is at least one $j\in I_n$
 such that  $\omega_j$ has 
 an interior, which by the open mapping theorem implies that $\Phi_j$ is a constant, whose value $c$ is in $\bd V$. For a fixed $s\in \cx^n$
 the roots of the equation  $q(s;t)=0$ are the coordinates of the points in $\pi^{-1}(s)$. It follows therefore that $q(\varphi(z);c)=0$ if $z\in \omega$. Since
 the function $z\mapsto q(\varphi(z);c)$ is holomorphic on $\Omega$, the result follows by the identity theorem.
\end{proof}

 \begin{proof}[Proof of Theorem~\ref{thm-proper}]
We proceed along the lines of the proof of the Remmert-Stein theorems  (cf. \cite[pp.~71--77]{nar}.) Consider a sequence 
$\{w^\nu\}\subset U_n$ such that $w^\nu\to w^0\in\bd U_n$. Denote by $G_{n-1}$ the domain $U_1\times\dots\times U_{n-1}$ in $\cx^{n-1}$, 
and define  the maps $\varphi^\nu$ on $G_{n-1}$ by 
\begin{equation}\label{eq-phinu}
\varphi^\nu('z)=f('z,w^\nu),
\end{equation}
where $'z=(z_1,\dots,z_{n-1})\in G_{n-1}$.
Since $f$ is bounded, the sequence $\{\varphi^\nu\}$ is  uniformly bounded. Hence, by Montel's theorem, 
there exists a subsequence $\{\nu_k\}_{k=1}^\infty$ such that $\varphi^{\nu_k}$ converges uniformly on compact subsets of $G_{n-1}$. 
Denote the limit by $\varphi:G_{n-1}\to \cx^n$. Since $f$ is a proper holomorphic map, $\varphi(G_{n-1})\subset \bd \Sigma^nV$.
Using Lemma~\ref{lem-bdy} we get that there is a constant $c\in\bd V$ such that on $G_{n-1}$:
\begin{equation}\label{eq-polyc}
q(\varphi;c)=c^n-c^{n-1}\varphi_1+\dots+(-1)^n\varphi_n\equiv 0.
\end{equation}
Differentiating \eqref{eq-polyc} with respect to $z_j$  for each $j\in I_{n-1}$ gives rise to the following system of equations:
\begin{equation}\label{eq-system}
\sum_{k=1}^n(-1)^kc^{n-k}\dfrac{\partial \varphi_k}{\partial z_j}=0,\quad  j\in I_{n-1}.
\end{equation}
Since $f: G\to \Sigma^n V$ is a proper holomorphic map,  the complex Jacobian determinant $\det(f')$  of $f$ does not vanish identically. 
Therefore, there is a $\mu\in I_n$  such the $(n-1)\times(n-1)$ minor of the Jacobian determinant  obtained by  omitting the 
$\mu$-th column is not identically zero on $G$.  We consider the following 
system of linear equations in the  $n-1$ variables  $(-1)^kc^{n-k},\quad  k\in I_n\setminus\{\mu\}$, obtained by moving the $\mu$-th column of  
\eqref{eq-system} to the 
right hand side:
\begin{equation}\label{eq-sysm-1}
\sum_{k\in I_n\setminus\{\mu\}}(-1)^kc^{n-k}\frac{\partial\varphi_k}{\partial z_j}=(-1)^{\mu+1}c^{n-\mu}
\frac{\partial \varphi_\mu}{\partial z_j},\quad  j\in I_{n-1}.
\end{equation}
Let $B$ be an $(n-1)\times n$ matrix, and let $k\in I_n\setminus\{ \mu\}$. Denote by $\Delta_k(B)$ the determinant of the $(n-1)\times (n-1)$
matrix obtained by  removing the  $\mu$-th column of $B$ and replacing it by the $k$-th column. Denote also by $\Delta_\mu(B)$ the determinant 
of the  $(n-1)\times (n-1)$ matrix obtained by  removing the  $\mu$-th column of $B$. Notice that each of the functions $\Delta_j$ is a polynomial in the 
entries of the matrix $B$.

 We now solve the system \eqref{eq-system} by Cramer's rule to obtain:
 \begin{equation}\label{eq-cramer}
  (-1)^k c^{n-k} = (-1)^{\mu} c^{n-\mu}\frac{\Delta_k({\mathsf D}_{n-1}\varphi)}{\Delta_\mu({\mathsf D}_{n-1}\varphi)},
  \end{equation}
where ${\mathsf D}_{n-1}\varphi$ is the $(n-1)\times n$ matrix 
$\left(\dfrac{\partial \varphi_k}{\partial z_j}\right)_{j\in I_{n-1}, k\in I_n}$. Note that by the choice of $\mu$, the denominator $\Delta_\mu({\mathsf D}_{n-1}\varphi)$ does not vanish identically on $G_{n-1}$.

We now consider two cases. If $\mu\not=1$,  we let $k=\mu-1$ in \eqref{eq-cramer}. Then we have on $G_{n-1}$ 
\begin{equation}\label{eq-cramer2} \Delta_{\mu-1}({\mathsf D}_{n-1}\varphi)= -c \cdot\Delta_{\mu}({\mathsf D}_{n-1}\varphi).\end{equation}

Let $\Psi_k(z)= \Delta_k({\mathsf D}_{n-1}\varphi(z)).$
 Then $\Psi_k$ is a holomorphic function on $G_{n-1}$. Note that $\Psi_k$ depends on the choice of the 
sequence $w^\nu$ as well as the subsequence $\nu_k$ (since $\varphi$ depends on these choices.) We have
$\Psi_{k-1}(z)= -c \Psi_k(z)$ for all  $z\in G_{n-1}$. Differentiating with respect to  $z_j$ and eliminating $c$  we obtain the relation
\begin{equation}\label{eq-psik}\Psi_{\mu-1}(z)\frac{\partial \Psi_\mu}{\partial z_j}(z)
- \Psi_{\mu}(z)\frac{\partial \Psi_{\mu-1}}{\partial z_j}(z)=0\end{equation}
for all $z\in G_{n-1}$.  Denote by  ${\mathsf D}_{n-1}^2\varphi$ the collection of  second derivatives of the components
$(\varphi_1,\dots, \varphi_n)$ of $\varphi$ with respect to the $(n-1)$ coordinate variables $z_1,\dots, z_{n-1}$.  From the chain rule, there is clearly 
a polynomial $E_{jk}$ such that
\begin{align*}\frac{\partial \Psi_k}{\partial z_j}(z)&= \frac{\partial}{\partial z_j}\Delta_k({\mathsf D}_{n-1}\varphi(z))\\
&= E_{jk}\left( {\mathsf D}_{n-1}\varphi(z); {\mathsf D}_{n-1}^2\varphi(z)\right).
\end{align*}
Therefore, if $H_{j}$ is the polynomial $\Psi_{\mu-1} E_{j\mu}- \Psi_\mu E_{j-1,\mu}$ then \eqref{eq-psik} may be rewritten as
\[ H_{j}\left({\mathsf D}_{n-1}\varphi(z);  {\mathsf D}_{n-1}^2\varphi(z)\right)=0.\]
As previously noted, the map $\varphi$ on the left hand side of this equation
depends on the sequence $w^\nu$ as well as the subsequence $\nu_k$, but the right hand side 
is independent of these choices.  Using Weierstrass theorem on the uniform convergence of derivatives, we conclude  that for each sequence
$\{w^\nu\}$ in $U_n$  converging  to a point in $\bd U_n$, we have
\[ \lim_{\nu\to \infty}H_{j}\left({\mathsf D}_{n-1} f(\cdot , w^\nu); {\mathsf D}_{n-1}^2 f(\cdot , w^\nu)\right)=0,\]
where ${\mathsf D}_{n-1} f$ and ${\mathsf D}^2_{n-1} f$ denote the collection of the first and second partial derivatives of $f$ with respect to the coordinates.
Consequently, we have by the maximum principle that  for $('z,w)\in G_{n-1}\times U_n=G$ 
\[ H_{j}\left({\mathsf D}_{n-1} f('z , w); {\mathsf D}_{n-1}^2 f('z , w)\right)=0,\]
i.e, on $G$ we have
\[ H_{j}\left({\mathsf D}_{n-1} f;{\mathsf D}_{n-1}^2 f\right)\equiv 0.\]
Recalling that this equation is a restatement of \eqref{eq-psik}, we conclude therefore that
\[ \Delta_{\mu-1}({\mathsf D}_{n-1} f) \frac{\partial}{\partial z_j} \Delta_{\mu}({\mathsf  D}_{n-1} f) 
- \Delta_{\mu}({\mathsf D}_{n-1} f) \frac{\partial}{\partial z_j} \Delta_{\mu-1}({\mathsf  D}_{n-1} f) =0,\]
which implies that
\[ \frac{\partial}{\partial z_j} \left(\frac{\Delta_{\mu-1}({\mathsf  D}_{n-1} f)}{\Delta_{\mu}({\mathsf  D}_{n-1} f) } \right) \equiv 0  \text{  for  } j \in I_{n-1},\]
holds outside the  proper analytic subset $A= \{ \Delta_{\mu}({\mathsf  D}_{n-1} f)=0\}$ of $G$. Therefore, there is a function $g_n$ defined on $G\setminus A$,
and depending only on $z_n$ such that on $G\setminus A$, we have
\[ g_n= -\frac{\Delta_{\mu-1}({\mathsf  D}_{n-1} f)}{\Delta_{\mu}({\mathsf  D}_{n-1} f)}.\]

Now let again $w^\nu$ be a sequence in $U_n$ converging to a point of $\bd U_n$, and let $\varphi$ and $c$ be the corresponding quantities as in 
\eqref{eq-polyc}. Taking setting $z_n=w^\nu$ in the above equation,  using relation \eqref{eq-cramer2} and taking the limit we see that $\lim_{\nu\to\infty} g_n(w^\nu)= c$. Again, in the function on $G$ given by $r(z)=q(f(z);g_n(z_n))$, setting $z_n=w^\nu$ and letting $\nu\to \infty$ shows that $r('z,w^\nu)\to 0$
as $\nu\to \infty$. It follows that $r('z,z_n)\to 0$ as $z_n\to \bd U_n$. Consequently, by the maximum principle $r\equiv 0$ on $G\setminus A$, i.e.,  we have
$q(f(z); g_n(z_n))\equiv 0$ on $G\setminus A$. Since $V$ is bounded,  it follows that the roots of the equation $q(f(z);t)=0$ are bounded independently of
$z\in G$, so that $g$ is a bounded holomorphic function, and therefore extends holomorphically to $G$. Consequently we have the relation $q(f(z); g_n(z_n))\equiv 0$ on the whole of $G$.

If $\mu=1$, we take $k=2$ in \eqref{eq-cramer}. One can repeat the arguments to conclude again that there is a holomorphic $g_n$ on $U_n$ such that 
$q(f(z); g_n(z_n))\equiv 0$ on $G$.

Repeating the argument for each factor $U_j$ in the product $G=U_1\times\dots\times U_n$, we conclude that there is a function $g_j:U_j\to V$ for 
each $j\in I_n$ such that  $q(f(z);g_j(z_j))= 0$ for $z$ in  $G$. It follows now that $\pi(g_1(z_1),\dots, g_n(z_n))=0$.
\end{proof}
We can now complete the proofs   of  Corollary~\ref{cor-proper} and  Corollary~\ref{cor-bdryreg}:
\begin{proof}[Proof of Corollary~\ref{cor-proper}]
Since $\Phi$ and $\pi$ are both proper holomorphic, $\Phi\circ\pi$ is proper holomorphic map from 
$U^n$ to $\Sigma^n V$.  From Theorem~\ref{thm-proper}, we conclude that  there exist proper holomorphic maps 
$\varphi_j:U\to V$, $j=1,\dots, n$, such that 
\[
\Phi\circ\pi=\pi\circ(\varphi_1\times\dots\times\varphi_n).
\]
We have $\Phi\circ\pi(z_1,\dots,z_n)=\pi(\varphi_1(z_1),\dots,\varphi_n(z_n))$, where $z=(z_1,\dots,z_n)\in U^n$. Hence, for each $\sigma\in S_n$
and  $z\in U^n$, we get, using the symmetry of the map $\pi$ that
\[
\pi(\varphi_1(z_1),\dots,\varphi_n(z_n))=\Phi\circ\pi(z)=\Phi\circ\pi(\sigma z)=\pi(\varphi_1(z_{\sigma(1)}),\dots,\varphi_n(z_{\sigma(n)})).
\]
This  is possible only when there is a $\varphi:U\to V$ such that
\[
\varphi_1=\varphi_2=\dots=\varphi_n =\varphi.
\]
Hence,  $\Phi=\Sigma^n\varphi$.
\end{proof}
\begin{proof}[Proof of Corollary~\ref{cor-bdryreg}] 
Since $\Phi:\Sigma^n U\to \Sigma^n V$ is a proper holomorphic map, by  Corollary~\ref{cor-proper}, 
there is a proper holomorphic map $\varphi:U\to V$ such that  $\Phi=\Sigma^n\varphi.$
On the other hand proper holomorphic maps between domains with $\mathcal{C}^\infty$-boundaries extend to $\mathcal{C}^\infty$ 
maps of the closures. This can be seen either from the general result in several variables (see \cite{DF, BC}), or in a more elementary way, 
by noting that the proper map being of finite degree, the branching locus is a finite set, and near the boundary the proper map is actually conformal,
so that the classical theory of boundary regularity in one variable may be used. The result now follows from Theorem~\ref{thm-akn}.
\end{proof}


\end{document}